\documentclass[10pt]{amsart}

\calclayout

\usepackage{amsmath, amssymb}
\usepackage{amsfonts}
\usepackage{mathrsfs}
\usepackage[color, arrow,matrix,curve,cmtip,ps]{xy}
\usepackage{paralist}
\usepackage{quiver}
\usepackage{amsthm}
\usepackage{caption}
\usepackage{tikz-cd}
\usetikzlibrary{arrows,calc,matrix}
\tikzset{
curvarr/.style={
  to path={ -- ([xshift=2ex]\tikztostart.east)
    |- (#1) [near end]\tikztonodes
    -| ([xshift=-2ex]\tikztotarget.west)
    -- (\tikztotarget)}
  }
}
\tikzset{%
    symbol/.style={%
        draw=none,
        every to/.append style={%
            edge node={node [sloped, allow upside down, auto=false]{$#1$}}}
    }
}

\usepackage[final]{pdfpages}
\usepackage{dsfont,txfonts}
\usepackage{tikz}
\usepackage{rotating}
\usepackage{mathrsfs}
\PassOptionsToPackage{hyphens}{url}\usepackage{hyperref}
\usepackage{enumitem}
\usepackage{graphicx}
\usepackage{mathtools}
\usetikzlibrary{arrows,chains,matrix,positioning,scopes}

\usepackage[T1]{fontenc}
\usepackage[variablett]{lmodern}
\usepackage{xcolor}
\usepackage[citestyle=alphabetic,bibstyle=alphabetic]{biblatex}
\usepackage{lipsum}

\newtheorem{theorem}{Theorem}[section]
\theoremstyle{definition}
\newtheorem{lemma}[theorem]{Lemma}

\newtheorem{proposition}[theorem]{Proposition}
\newtheorem{corollary}[theorem]{Corollary}
\newtheorem{definition}[theorem]{Definition}

\newtheorem{remark}[theorem]{Remark}

\newtheorem{question}[theorem]{Question}

\newtheorem{example}[theorem]{Example}

\everymath{\displaystyle}
\allowdisplaybreaks

\newcommand{\Fun}{\text{Fun}}

\newcommand{\colim}{\text{colim}}

\DeclareFieldFormat{postnote}{#1}
\DeclareFieldFormat{multipostnote}{#1}

\begin{document}
\title{Algebraic $K$-theory of coherent spaces}
\author{Georg Lehner}
\begin{abstract}
We give a description of the value of a finitary localizing invariant, such as algebraic $K$-theory, on the category of sheaves on a locally coherent space $X$. This in particular includes all spaces that arise as spectra of commutative rings. As applications we discuss the connection between scissors congruence $K$-theory and Topological Hochschild Homology of certain locally coherent spaces, as well as the algebraic $K$-theory of a measure space.
\end{abstract}
\maketitle

\begin{figure}[h!]
    \centering
    \includegraphics[width=0.45\textwidth]{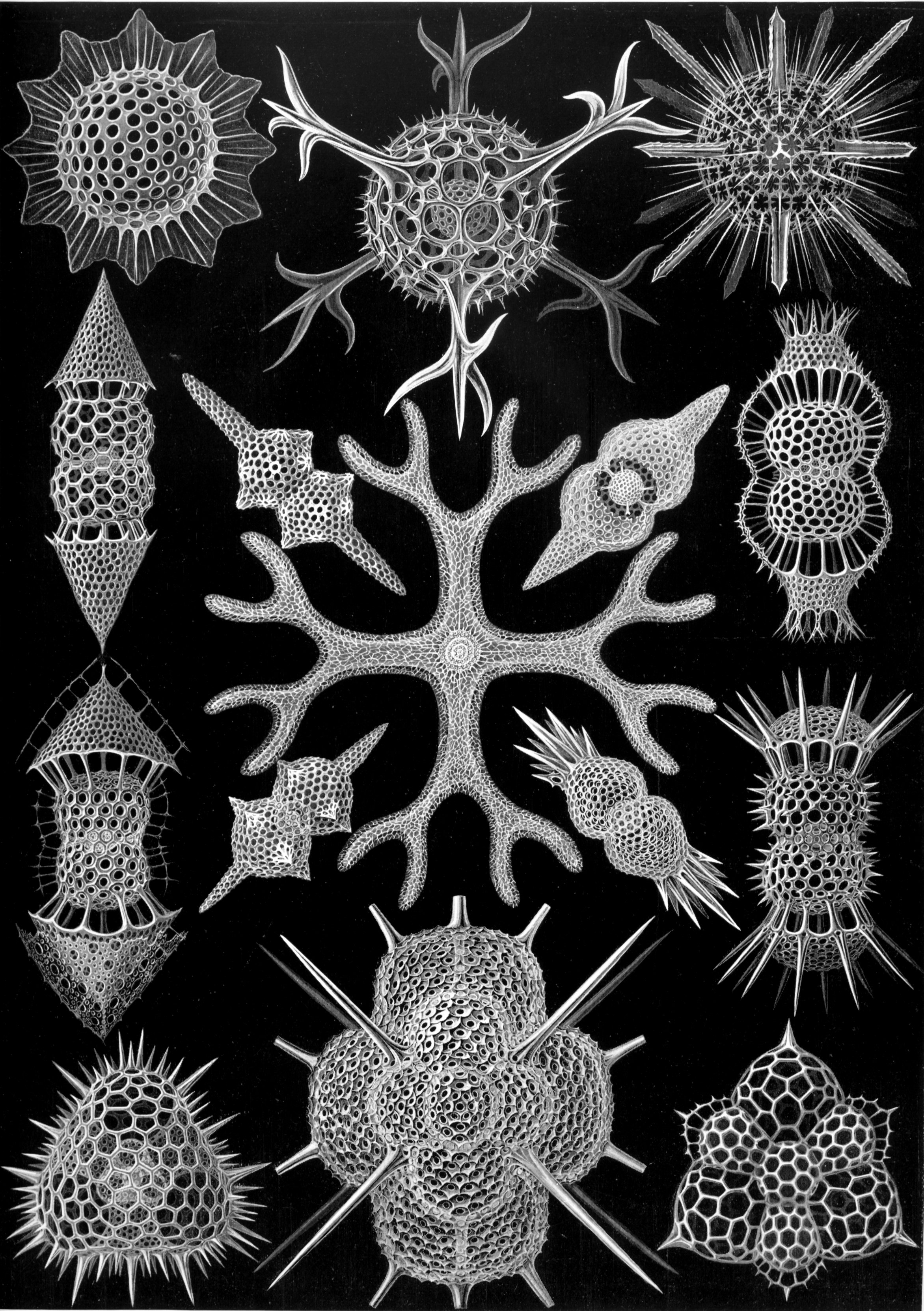}
    \caption*{\tiny Ernst Haeckel, \textit{Kunstformen der Natur}, 1904, plate 91: Spumellaria, Public domain, via Wikimedia Commons}
\end{figure}

\tableofcontents

\section{Introduction}

Recent progress made by Efimov \cite{efimov2025ktheorylocalizinginvariantslarge}, Krause-Nikolaus-Pützstück \cite{krause_nikolaus_puetzstueck}, Lurie, and Clausen among others have raised interest in the (continuous) $K$-theory of a class of large stable $\infty$-categories called \emph{dualizable} $\infty$-categories. An example is the $\infty$-category of sheaves on a locally compact Hausdorff space $X$. The result that is central here is the computation of the $K$-theory of this category.

\begin{theorem}[\cite{efimov2025ktheorylocalizinginvariantslarge} Theorem 6.11, see also \cite{krause_nikolaus_puetzstueck}  Theorem 3.6.1] \label{ktheorylch}
Let $X$ be a locally compact Hausdorff space. Then
$$K^{cont}( \mathrm{Sh}(X, \mathrm{Sp}) ) \simeq H_{cs}(X; K(\mathbb{S}) )$$
where $\mathrm{Sp}$ is the $\infty$-category of spectra and the right-hand side refers to compactly supported sheaf cohomology of $X$ with respect to the local system given by the $K$-theory of the sphere spectrum. 
\end{theorem}
This result is used as a starting point to import point-set topological techniques into the study of algebraic $K$-theory, not unlike one does in the setting of operator algebras and their (topological) $K$-theory. However, locally compact Hausdorff spaces are not the only class of spaces whose sheaf categories are dualizable $\infty$-categories. Another class of spaces whose $\infty$-categories of sheaves are dualizable is the class of \emph{locally coherent spaces}.

\begin{definition}
A space $X$ is called \emph{coherent} (sometimes also called a \emph{spectral space}) if either of the following equivalent statements holds:
\begin{itemize}
\item $X$ is a compact sober space together with a basis of compact open sets, closed under finite intersections.
\item $X$ is obtained as the spectrum of a commutative ring.
\item $X$ is obtained as an inverse limit of a diagram of finite posets equipped with the Alexandroff topology.
\item The frame of open sets of $X$ is obtained as the Ind-completion of a bounded distributive lattice.
\end{itemize}
A space $X$ is called \emph{locally coherent} if either of the following holds:
\begin{itemize}
\item $X$ is a sober space together with a basis of compact open sets, closed under finite intersections.
\item The frame of open sets of $X$ is obtained as the Ind-completion of a lower bounded distributive lattice.
\item The frame of open sets of $X$ is obtained from a finitary Grothendieck topology on a poset closed under finite meets.
\end{itemize}
\end{definition}
\noindent We note that a coherent space $X$ is Hausdorff iff $X$ is a profinite set, also called a Stone space.

There is a Stone duality relating the category of locally coherent spaces with the category of lower bounded distributive lattices. Under this duality, a locally coherent space $X$ is sent to the lattice $\mathcal{K}^o(X)$ of compact open subsets. We have the equivalence
$$ \mathrm{Sh}(X,\mathcal{C}) \simeq  \mathrm{Sh}((\mathcal{K}^o(X),fin),\mathcal{C}) $$
for any presentable $\infty$-category $\mathcal{C}$. Here, for the right-hand side we equip the poset of compact opens with the Grothendieck topology obtained by taking finite joins as coverings. Locally coherent spaces are special in the sense that if we choose a compactly generated base $\infty$-category $\mathcal{C}$, the $\infty$-category of sheaves $\mathrm{Sh}(X,\mathcal{C})$ is again compactly generated. The central result that allows the computation of the $K$-theory of this category is the following. Denote the category of lower bounded distributive lattices as $\mathrm{DLatt}_{lb}$.

\begin{theorem}[See Theorem \ref{sheavesfiltereddualizablelowerbounded}]
Let $\mathcal{C}$ be a dualizable, stable $\infty$-category. The functor
$$ \mathrm{Sh}((-,fin),\mathcal{C}) : \mathrm{DLatt}_{lb} \rightarrow \mathrm{Pr}^L_{dual} $$
preserves filtered colimits.
\end{theorem}

Here with $\mathrm{Pr}^L_{dual}$ we mean the $\infty$-category of stable, dualizable $\infty$-categories. Any finitary localizing invariant $F : \mathrm{Cat}^{perf} \rightarrow \mathcal{E}$ in the sense of \cite{blumberg_gepner_tabuada}, such as for example non-connective algebraic $K$-theory or topological Hochschild homology has a unique extension $F^{cont} : \mathrm{Pr}^L_{dual} \rightarrow \mathcal{E}$, such that for compactly generated stable $\infty$-categories $\mathcal{C}$ we have $F^{cont}( \mathcal{C} ) = F( \mathcal{C}^\omega )$. \cite{efimov2025ktheorylocalizinginvariantslarge} We obtain the useful corollary.

\begin{corollary}
Let $F$ be a finitary localizing invariant $ \mathrm{Cat}^{perf} \rightarrow  \mathcal{E}$, such as $K$-theory or $THH$. Then the functor
$$ F^{cont}( \mathrm{Sh}((-,fin),\mathcal{C})) : \mathrm{DLatt}_{lb} \rightarrow \mathcal{E}$$
preserves filtered colimits.
\end{corollary}

This means the computation of the $K$-theory of these categories can be reduced to the case of finite distributive lattices, which under Birkhoff's theorem correspond to Alexandroff topologies on finite posets. We use this to give two different characterizations of the value of $\mathrm{Sh}(X,\mathcal{C})$ under a finitary localizing invariant. For one, for any coherent space $X$, there is a natural map $X^{const} \rightarrow X$ that equips $X$ with a profinite topology, called the \emph{constructible topology}.

\begin{theorem}[See Theorem \ref{ktheorycoherent} and Corollary \ref{ktheorycoherentvaluesspectra}]
Let $X$ be a coherent space. The natural map $X^{const} \rightarrow X$ induces an equivalence
$$F^{cont}( \mathrm{Sh}(X,\mathcal{C})) \simeq F^{cont}( \mathrm{Sh}(X^{const},\mathcal{C})) $$
for any finitary localizing invariant $F$. In particular, if $F$ has values in spectra, then for all $n \in \mathbb{Z}$ we have
$$ \pi_n F^{cont}( \mathrm{Sh}(X,\mathcal{C})) \cong C(X^{const}; \pi_n(F^{cont}(\mathcal{C}) )$$
where $\pi_n(F^{cont}(\mathcal{C}))$ is equipped with the discrete topology.
\end{theorem}

This means the computation of the $K$-theory of these categories reduces to the case of compact Hausdorff spaces, assuming one understands the space corresponding to the constructible topology reasonably well. Since this might not always be useful in practice, we give a different description in terms of valuations.

\begin{definition}
Let $D$ be lower bounded distributive lattice and $A$ an abelian group. An $A$-valued valuation on $D$ is a function $\mu : D \rightarrow A$ such that
\begin{enumerate}
\item $\mu(0) = 0$
\item \emph{Modularity:} For all $U, V \in D$ it holds that $\mu(U) + \mu(V) = \mu(U \vee V) + \mu(U \wedge V)$.
\end{enumerate}
Given a lower bounded distributive lattice $D$, there exists a universal valuation $\mu_{univ} : D \rightarrow M(D)$ defined by the property that whenever $\mu : D \rightarrow A$ is a valuation, there exists a unique group homomorphism $M(D) \rightarrow A$ extending $\mu$ along $\mu_{univ}$. We call $M(D)$ the \emph{module of D-motives}. The group $M(D)$ can alternatively be described as the free abelian group $\mathbb{Z}[D]$ modulo the relations:
\begin{enumerate}
\item $[0] = 0$
\item $[U] + [V] = [U \vee V] + [U \wedge V]$ for all $U, V \in D$.
\end{enumerate}
Write $\mathcal{M}(D)$ for the Moore spectrum corresponding to $M(D)$. We call this $\mathcal{M}(D)$ the spectrum of motives.
\end{definition}

The abelian group $M(D)$ is always a free group, which we show in Section \ref{sectionvaluations}, by reducing the statement to Nöbeling's theorem, and hence $\mathcal{M}(D)$ is in fact a wedge of sphere spectra indexed by a choice of basis of $M(D)$. The assignment of an abelian group to its Moore spectrum is usually not functorial; however, the issue of functoriality of the assignment $D \mapsto \mathcal{M}(D)$ can be clarified quite substantially by the following result, which we develop in Section \ref{functoriality}. 

\begin{theorem}[See Theorem \ref{spaceofidempotents}] \label{spaceofidempotentsintro}
There is an adjunction
\[\begin{tikzcd}
	{\mathrm{DLatt}_{lb}} & {\mathrm{CAlg}^{nu}(\mathrm{Sp})}
	\arrow[""{name=0, anchor=center, inner sep=0}, "{\mathcal{M}}", curve={height=-12pt}, from=1-1, to=1-2]
	\arrow[""{name=1, anchor=center, inner sep=0}, "{\mathrm{Idem}}", curve={height=-12pt}, from=1-2, to=1-1]
	\arrow["\dashv"{anchor=center, rotate=-90}, draw=none, from=0, to=1]
\end{tikzcd}\]
such that for a lower bounded distributive lattice $D$, the value $\mathcal{M}(D)$ is the Moore spectrum to the module of motives $M(D)$.
\end{theorem}

The question of when a commutative ring $R$ lifts to an $E_\infty$-ring spectrum with underlying spectrum given by the Moore spectrum of $R$ is of independent interest, see e.g.\ \cite{carmeli2024mapssphericalgrouprings}. Theorem \ref{spaceofidempotentsintro} provides us with many such examples. With that being said, let us now state the main theorem.

\begin{theorem} \label{ktheorylocallycoherent}
Let $X$ be a locally coherent space. Then
$$F^{cont}( \mathrm{Sh}(X,\mathcal{C})) \simeq \mathcal{M}(\mathcal{K}^o(X)) \otimes F^{cont}(\mathcal{C}) \simeq \bigoplus F^{cont}(\mathcal{C})$$
for any finitary localizing invariant $F$, where $\mathcal{K}^o(X)$ is the lower bounded distributive lattice of compact open subsets of $X$, and the wedge sum on the right is indexed by a choice of basis of $M(\mathcal{K}^o(X))$. In particular, if $F$ has values in spectra, then for all $n \in \mathbb{Z}$ we have
$$ \pi_n F^{cont}( \mathrm{Sh}(X,\mathcal{C})) \cong M(\mathcal{K}^o(X)) \otimes_\mathbb{Z} \pi_n(F^{cont}(\mathcal{C}) ).$$
\end{theorem}

\begin{remark}
The above results indicate that locally coherent spaces can be thought of as zero dimensional objects from the perspective of $K$-theory. In some sense, this is disappointing, as the homotopy theory of these spaces, even just for finite posets, is very rich. (E.g. every homotopy type of a finite $CW$-complex can be obtained from a finite poset, see \cite{McCord_finite_topological}.) It also means that no straightforward extension of Theorem \ref{ktheorylch} exists in the non-Hausdorff context, at least not framed in terms of sheaf cohomology.
\end{remark}

\begin{remark}
We note that if $\mathcal{C}$ is even compactly generated, the usage of dualizable $\infty$-categories could in theory be ignored, as the category $\mathrm{Sh}(X,\mathcal{C})$ is then also compactly generated for $X$ locally coherent and we have
$$ F^{cont}( \mathrm{Sh}((-,fin),\mathcal{C})) \simeq F( \mathrm{Sh}((-,fin),\mathcal{C})^{\omega}).$$
However, the formal theory of dualizable $\infty$-categories allows for a much cleaner, conceptual proof than a corresponding argument using the category $\mathrm{Cat}^{perf}$ instead.
\end{remark}

\subsection{Motivation}

When it comes to the study of the algebraic $K$-theory of $\infty$-categories of sheaves on spaces, locally coherent spaces play a special role. The largest natural class of spaces $X$ for which the $\infty$-category $\mathrm{Sh}(X)$ is compactly assembled (and hence its stabilization is dualizable), is that of \emph{stably locally compact spaces}. (See \cite[Corollary 3.3.3.]{anel_lejay_exponentiable}.) A space $X$ called stably locally compact if $X$ is sober, $\Omega(X)$ is a continuous poset and the way-below relation $U \ll V$ between opens $U,V$ is stable under intersection.  This includes all locally compact Hausdorff spaces. A space is stably locally compact iff it is a proper quotient of a locally coherent space, hence locally coherent spaces can be thought of as ``free'' spaces from the perspective of algebraic $K$-theory. The choice of such a quotient map can in fact be made naturally. If $X$ is stably locally compact, then there is a natural double adjunction
\[\begin{tikzcd}
	{\Omega(X)} & {\mathrm{Ind}(\Omega(X)),}
	\arrow[""{name=0, anchor=center, inner sep=0}, "{\hat{y}}", curve={height=-12pt}, from=1-1, to=1-2]
	\arrow[""{name=0p, anchor=center, inner sep=0}, phantom, from=1-1, to=1-2, start anchor=center, end anchor=center, curve={height=-12pt}]
	\arrow[""{name=1, anchor=center, inner sep=0}, curve={height=12pt}, from=1-1, to=1-2]
	\arrow[""{name=1p, anchor=center, inner sep=0}, phantom, from=1-1, to=1-2, start anchor=center, end anchor=center, curve={height=12pt}]
	\arrow[""{name=2, anchor=center, inner sep=0}, from=1-2, to=1-1]
	\arrow[""{name=2p, anchor=center, inner sep=0}, phantom, from=1-2, to=1-1, start anchor=center, end anchor=center]
	\arrow[""{name=2p, anchor=center, inner sep=0}, phantom, from=1-2, to=1-1, start anchor=center, end anchor=center]
	\arrow["\dashv"{anchor=center, rotate=-90}, draw=none, from=0p, to=2p]
	\arrow["\dashv"{anchor=center, rotate=-90}, draw=none, from=2p, to=1p]
\end{tikzcd}\]
where $\Omega(X)$ is the frame of opens of $X$. The functor $\hat{y}$ only preserves binary meets and not necessarily the top element, but becomes a frame homomorphism when considered as a functor to
$$\mathrm{Ind}(\Omega(X))_{/ \hat{y}(1) } \cong \mathrm{Ind}(\Omega(X)_{\ll}),$$
where $\Omega(X)_{\ll}$ is the lower bounded distributive lattice of all opens $U$ such that $U \ll 1$. This translates under Stone duality to a proper continuous map from the associated locally coherent space corresponding to $\Omega(X)_{\ll}$ into $X$. On $K$-theory, this means that there is a natural map
$$K( \mathrm{Sh}(X ; \mathcal{C}) ) \rightarrow K(  \mathrm{Sh}((\Omega(X)_{\ll}, fin) , \mathcal{C}) ) \simeq \mathcal{M}( \Omega(X)_{\ll} ) \otimes K( \mathcal{C} ).$$
In other words, we can think of the computation of the $K$-theory of locally coherent spaces as a building block for a deeper understanding of the $K$-theory of categories of sheaves on spaces in general. Beyond this abstract justification, we want to provide two more concrete and interesting examples.

\begin{example}
A highly interesting example appears in the context of scissors congruence $K$-theory, as developed by Zakharevich \cite{zakharevich2011scissorscongruencektheory}. We follow the description given in \cite{malkiewich2024higherscissorscongruence}. Fix a classical geometry $X = E^n, S^n$ or $H^n$ (euclidean, spherical or hyperbolic) and consider the collection $D(X)$ of $n$-dimensional polytopes in $X$. Then $D(X)$ is a lower bounded distributive lattice, and we can consider the associated locally coherent space $X_{\mathrm{Poly}}$ under Stone duality. This space comes with a proper, continuous and surjective map
$$X_{\mathrm{Poly}} \rightarrow X$$
which we will refer to as the \emph{universal cutting space} over $X$. (It effectively slices $X$ at each point along any polytope to achieve a totally disconnected space.) The abelian group $M(D(X))$ is referred to as $\mathrm{Pt}(X)$ and called the \emph{polytope module}. It can be shown that $\mathrm{Pt}(X)$ is a free abelian group using Nöbeling's theorem \cite{NÖBELING1968/69}, see also \cite{asgeirsson:LIPIcs.ITP.2024.6}. A computation due to Malkiewich and Zakharevich (See \cite[ Theorem 1.10]{malkiewich2024higherscissorscongruence}) shows that the non-equivariant scissors congruence $K$-theory splits as
$$\mathcal{K}^{sci}(X,1) \simeq \bigoplus \mathbb{S}$$
where the sum is taken over a choice of basis of $\mathrm{Pt}(X)$. (The $1$ refers to the fact that we are considering the \emph{non-equivariant} situation.)  It is an immediate consequence of Theorem \ref{ktheorylocallycoherent} that we have equivalences
$$ \mathrm{THH}( \mathrm{Sh}(X_{\mathrm{Poly}},\mathrm{Sp}) ) \simeq \mathcal{M}(D(X)) \simeq \mathcal{K}^{sci}(X,1).$$
This (non-equivariant) result suggests that scissors congruence $K$-theory, or more generally the $K$-theory of assemblers, could be thought of as being part of the study of $\mathrm{THH}$ of dualizable categories, thus potentially linking two distinct fields. We discuss this in more detail in Section \ref{scissorscongr}.
\end{example}

\begin{example}
Given any measurable locale $L$, represented by a complete Boolean algebra $B$ in the sense of \cite{Pavlov_2022}, for example obtained from a (nice enough) measure space $(X,\mathcal{F}, \mu)$, we have an associated hyperstonean space $\mathrm{Stone}(L)$ (its locale is simply obtained by considering the frame $\mathrm{Ind}(B)$). Then the dual of $L^\infty(L) = C^b(L;\mathbb{C})$ is given by the Banach space of bounded, finitely additive, complex valued valuations on $M$, see \cite{Toland2020}. As such we have that 
$$L^\infty(L)^* \subset \mathrm{Hom}_{\mathrm{Ab}}( M(B), \mathbb{C} )$$
can be identified as a subset. The fact that
$$K_0(\mathrm{Sh}(\mathrm{Stone}(L);\mathrm{Sp}) ) \cong M(B)$$
is itself the universal recipient of valuations on $B$
suggests a tight analogy between the von Neumann algebra $L^\infty(L)$ and the dualizable $\infty$-category $\mathrm{Sh}(\mathrm{Stone}(L);\mathrm{Sp})$. For a longer discussion see Section \ref{algebraicktheorymeasure}. An extension of this connection to not necessarily commutative von Neumann algebras and a certain class of dualizable $\infty$-categories seems plausible and should be explored in future work.
\end{example}

\subsection{Acknowledgements}

The author thanks Thomas Nikolaus, Maxime Ramzi, Marc Hoyois, Phil Pützstück, Benjamin Dünzinger, Holger Reich, Chris Huggle and David Wärn for feedback and discussions that were helpful for the results of this paper. Special thanks goes to Cary Malkiewich for his explanations on scissors congruence $K$-theory.

The author would also like to thank the Isaac Newton Institute for Mathematical Sciences, Cambridge, for support and hospitality during the programme New horizons for equivariance in homotopy theory, where work on this paper was undertaken. This work was supported by EPSRC grant EP/Z000580/1.

\section{Locales and Grothendieck topologies on posets}

We will use the language of frames and locales throughout this paper. For general resources, see \cite{johnstone1982stone} and \cite{picado_pultr}. A join-semilattice $(D, \leq)$ is a poset closed under finite joins (= coproducts). A distributive lattice $(D,\leq)$ is a poset that admits joins of two elements, written as $U \vee V$, and meets, written as $U \wedge V$, and the relation
$$U \wedge (V \vee W) = (U \wedge V) \vee (U \wedge W)$$
holds, for any $U,V,W \in D$. A distributive lattice is called \emph{lower bounded} if it has a bottom element $0$, and furthermore \emph{bounded} if it also has a top element $1$.
A functor $f : D \rightarrow D'$ between distributive lattices is called a \emph{lattice homomorphism} if $f$ preserves meet and joins. We denote the categories of lower bounded, respectively bounded, distributive lattices by
$$ \mathrm{DLatt}_{lb} ~\text{ and }~ \mathrm{DLatt}_{bd},$$
where morphisms are lattice homomorphisms that preserve the bottom element, respectively the bottom and top element.

A \emph{frame} is a poset $(F, \leq )$ which is cartesian closed and presentable when viewed as a category. We note that in particular any frame is a bounded distributive lattice. A \emph{frame homomorphism} $f^* : F \rightarrow F'$ is a functor preserving colimits and finite limits. This defines the category $\mathrm{Frm}$ of frames and frame homomorphisms. The category of locales is defined as the opposite category $\mathrm{Loc} = \mathrm{Frm}^{op}$. We refer to the corresponding frame $F$ of a locale $L$ as the \emph{frame of opens} of $L$, and sometimes use the notation $F = \Omega(L)$. There exists an adjunction, \cite[II.4.6]{picado_pultr},
\[\begin{tikzcd}
	{\mathrm{Top}} & {\mathrm{Loc}}
	\arrow[""{name=0, anchor=center, inner sep=0}, "\Omega", curve={height=-6pt}, from=1-1, to=1-2]
	\arrow[""{name=1, anchor=center, inner sep=0}, "{\mathrm{pts}}", curve={height=-6pt}, from=1-2, to=1-1]
	\arrow["\dashv"{anchor=center, rotate=-90}, draw=none, from=0, to=1]
\end{tikzcd}\]
where the left adjoint $\Omega$ sends a topological space $X$ to the frame of open sets of $X$. This adjunction becomes an equivalence on the full subcategories of sober spaces and spatial locales.

The functor $\mathrm{pts}$ can be understood in the following way. Let $\mathbf{2} = \{0 \leq 1\}$. Clearly, $\mathbf{2} = \Omega( \mathrm{pt})$, where $\mathrm{pt}$ is the one-point space. Given a locale $L$ with associated frame $F$, the set of points is defined as
$$ \mathrm{pts}(L) = \mathrm{Hom}_{\mathrm{Loc}}(\mathrm{pt},L) = \mathrm{Hom}_{\mathrm{Frm}}(F,\mathbf{2}).$$
It is equipped with a natural topology, by associating to opens $U$ of $L$ the corresponding sets $\Sigma_U = \{ p \in \mathrm{pts}(L) ~|~ U \in p^*(1) \}$. Since a frame homomorphism $p^* : F \rightarrow \mathbf{2}$ is completely determined by the set $p^*(1) \subset F$, points can be equivalently described by \emph{completely prime filters}. These are subsets $\mathcal{F} \subset F$ such that:
\begin{enumerate}
\item $\mathcal{F}$ is upward closed, i.e. if $U \in \mathcal{F}$, $U \leq V$, then $V \in \mathcal{F}$.
\item $\mathcal{F}$ is closed under meets, i.e. if $U \in \mathcal{F}$, $V \in \mathcal{F}$, then $U \wedge V \in \mathcal{F}$.
\item $\mathcal{F}$ is completely prime, i.e. if $\bigvee_{i \in I} U_i \in \mathcal{F}$, then there exists $i \in I$ such that $U_i \in \mathcal{F}$.
\end{enumerate}
For details, see \cite[II.3+4]{picado_pultr}. Given any locale $L$, represented by a frame $F$, and an open $U$, there exist two associated frames $F_{/U}$ and $F_{U/}$, the \emph{open}, respectively \emph{closed}, sublocales of $L$ associated to $U$. These come with adjunctions
\[\begin{tikzcd}
	{F_{/U}} & F & {F_{U/}}
	\arrow[""{name=0, anchor=center, inner sep=0}, "{i_!}", curve={height=-12pt}, hook, from=1-1, to=1-2]
	\arrow[""{name=1, anchor=center, inner sep=0}, "{i_*}"', curve={height=12pt}, hook, from=1-1, to=1-2]
	\arrow[""{name=2, anchor=center, inner sep=0}, "{i^*}"{description}, from=1-2, to=1-1]
	\arrow[""{name=3, anchor=center, inner sep=0}, "{c^*}", curve={height=-12pt}, from=1-2, to=1-3]
	\arrow[""{name=4, anchor=center, inner sep=0}, "{c_*}"{description}, hook', from=1-3, to=1-2]
	\arrow["\dashv"{anchor=center, rotate=-90}, draw=none, from=0, to=2]
	\arrow["\dashv"{anchor=center, rotate=-90}, draw=none, from=2, to=1]
	\arrow["\dashv"{anchor=center, rotate=-90}, draw=none, from=3, to=4]
\end{tikzcd}\]
where $i^* = - \wedge U$ and $c^* = - \vee U$ are frame homomorphisms, and the functor $c_*$ preserves all non-empty suprema. This generalizes the well-known open-closed decomposition of a topological space $X$, given an open subset and its closed complement.

Locales can be thought of as $0$-topoi, and just like topoi arise naturally from Grothendieck topologies. If $P$ is a poset, we denote the meet (in other words product) of two elements $p,q \in P$ by $p \wedge q$, if it exists.

\begin{definition}
Let $(P,\leq)$ be a poset with binary meets. A \emph{Grothendieck pretopology} $\tau$ on $P$ is for each $p \in P$ a collection of subsets of $P_{/p}$ called \emph{coverings}. We use the notation $\{p_i \leq p ~|~i \in I\}$ for a such a subset of $P_{/p}$. Moreover, the collection of all coverings needs to satisfy the conditions:
\begin{itemize}
\item \emph{Identities:} For all $p \in P$, the set $\{ p \leq p \}$ is a covering.
\item \emph{Stability under base change:} If $\{p_i \leq p ~|~i \in I\}$ is a covering and $q \leq p$, then $\{p_i \wedge q \leq q ~|~i \in I\}$ is a covering.
\item \emph{Locality:} If $\{p_i \leq p ~|~i \in I\}$ is a covering and $\{p_{ij} \leq p_i ~|~j \in J_i\}$ is a covering for each $i \in I$, then $\{p_{ij} \leq p ~|~i \in I,j \in J_i\}$ is a covering.
\end{itemize}
We call $(P,\leq)$ together with a Grothendieck pretopology $\tau$ a \emph{locally cartesian 0-site.}
We further call $(P, \tau)$ \emph{coherent}, if all coverings in $\tau$ are finite.
\end{definition}

\begin{example}
A central example will be the coherent site $(D,fin)$ obtained from a lower bounded distributive lattice $D$ with coverings $\{d_i \leq d ~|~ i \in I\}$ whenever $I$ finite and $d = \bigvee_{i \in I} d_i \in D$.
\end{example}

The definition of a Grothendieck pretopology on a poset is of course a special case of that of a Grothendieck pretopology on a category, but it will suffice for the purposes of this paper.\footnote{The requirement that finite meets exist could be relaxed, but this would cause issues later on as soon as we deal with sheaves valued in the $\infty$-category of spaces.} Given any locally cartesian $0$-site $(P,\tau)$, we have an associated locale, given by the frame of \emph{truth-valued or propositional sheaves} on $P$.

\begin{definition} Let $(P,\tau)$ be a locally cartesian $0$-site. A functor $f: P^{op} \rightarrow \mathbf{2}$ is called a \emph{propositional sheaf} if for all coverings $\{p_i \leq p ~|~i \in I\}$ we have
$$f(p) = \bigwedge_{i \in I} f(p_i).$$
The locale $L(P,\tau)$, generated by $(P,\tau)$, is defined by the corresponding frame given by the sub-poset
$$\mathrm{Sh}((P,\tau), \mathbf{2}) \subset \mathrm{Fun}(P^{op}, \mathbf{2})$$
spanned by propositional sheaves.
\end{definition}

Any propositional sheaf $f : P^{op} \rightarrow \mathbf{2}$ is determined by the downward closed set $f^{-1}(1) \subset P$. Conversely, a downward closed set corresponds to a propositional sheaf iff it is also closed under coverings. The inclusion
$$\mathrm{Sh}((P,\tau), \mathbf{2}) \subset \mathrm{Fun}(P^{op}, \mathbf{2})$$
preserves limits, hence there exists a ``propositional sheafification'' left adjoint $(-)^{sh}$, which sends a downward closed subset $\mathcal{F} \subset P$ to its closure under coverings, i.e.
$$\mathcal{F}^{sh} = \{p \in P ~|~ \exists \text{ covering } \{p_i \leq p ~|~ i \in I \} \text{ s.t. } p_i \in \mathcal{F} \text{ for all } i \in I \}.$$
It can be checked that $\mathcal{F}^{sh}$ is the smallest propositional sheaf containing $\mathcal{F}$.

\begin{lemma} \label{leftadjointgeometricmorphism}
The functor $(-)^{sh} : \mathrm{Fun}(P^{op}, \mathbf{2}) \rightarrow \mathrm{Sh}((P,\tau), \mathbf{2})$
preserves finite meets.
\end{lemma}

\begin{proof} Note that meets in both posets are just given by intersection. It is clear that the terminal object, which corresponds to the downward closed subset given by $P$ itself is already a propositional sheaf, so $1^{sh} = 1$. Now let $\mathcal{F}, \mathcal{G} \subset P$ be downward closed subsets. We have
$$ ( \mathcal{F}  \wedge \mathcal{G} )^{sh} \leq
\mathcal{F}^{sh} \wedge \mathcal{G}^{sh}  $$
simply because $(-)^{sh}$ is a functor. Let $p \in \mathcal{F}^{sh} \wedge \mathcal{G}^{sh}$, which is equivalent to saying that there are coverings $\{p_i \leq p ~|~ i \in I \}$ and $\{q_j \leq p ~|~ j \in J \}$, with $p_i \in \mathcal{F}$ for all $i \in I$ and $q_j \in \mathcal{G}$ for all $j \in J$. Then the set  $\{p_i \wedge q_j \leq p ~|~ i \in I, j \in J \}$ is a covering with $p_i \wedge q_j \in \mathcal{F}  \wedge \mathcal{G}$, by application of stability under base change and locality.
\end{proof}

It follows that $\mathrm{Sh}((P,\tau), \mathbf{2})$ is a frame and that $(-)^{sh} : \mathrm{Fun}(P^{op}, \mathbf{2}) \rightarrow \mathrm{Sh}((P,\tau), \mathbf{2})$ is a homomorphism of frames. The following is a standard result in locale theory that we wanted to spell out in detail.

\begin{proposition}[Basis theorem] \label{basistheorem}
Let $F$ be a frame and $P \subset F$ a subset closed under meets, such that for all $U \in F$ there exists $B_i \in P$ such that $U = \bigvee B_i$ in $F$. We call such $P$ a \emph{basis} for $F$.  Equip $P$ with the Grothendieck pretopology induced from $F$, that is $\{B_i \leq B ~|~ i \in I \}$ is a covering iff $\bigvee_{i \in I} B_i = B$ in $F$. Then the functor
$$\begin{array}{rcl} h : F & \rightarrow & \mathrm{Sh}((P,\tau), \mathbf{2}) \\
U & \mapsto & h_U = \mathrm{Hom}_F( -, U)
\end{array}$$
is an isomorphism.
\end{proposition}

\begin{proof}
Let us first verify that $h$ is well-defined, i.e. for all $U \in F$, the functor $h_U$ satisfies the sheaf condition. Let $\{B_i \leq B ~|~ i \in I \}$ be a covering; in other words, $\bigvee_{i \in I} B_i = B$. Then we have the following chain of equivalent statements:
\[\begin{array}{rccccccl}
 & h_U(B_i) = 1 \text{ for all } i \in I &
\text{iff} & B_i \leq U \text{ for all } i \in I &
\text{iff} & B = \bigvee_{i \in I} B_i \leq U &
\text{iff} & h_U(B) = 1.
\end{array}\]
Now define the functor
$$\begin{array}{rcl}
\varphi: \mathrm{Sh}((P,\tau), \mathbf{2}) &\rightarrow & F \\
f & \mapsto &\bigvee f^{-1}(1).
\end{array}$$
Let us show that $\varphi$ and $h$ are inverse to another. Let $U \in F$. Then 
$$\varphi(h_U) = \bigvee_{B \leq U, ~ B \in P} B = U$$
since any open is obtained as a union of elements of $P$. Conversely, let $f$ be a propositional sheaf. Recall $\varphi(f) = \bigvee f^{-1}(1) \in F$. 
Let $B \in P$. Then 
\[\begin{array}{rcccl}
f(B) = 1 &
\text{iff} & B \leq \bigvee f^{-1}(1) &
\text{iff} & h_{\varphi(f)}(B) = 1,
\end{array}\]
hence $f = h_{\varphi(f)}$.
\end{proof}

\begin{corollary}
The frame $F$ of any locale is obtained as the frame of propositional sheaves on a locally cartesian $0$-site.
\end{corollary}

\begin{proof} Choose $B = F$ in the basis theorem.
\end{proof}

\begin{example} \label{indfinitetoplogy} Let $D$ be a lower bounded distributive lattice. Then $\mathrm{Ind}(D)$ is a frame. It is clear that $y : D \rightarrow \mathrm{Ind}(D)$ identifies $D$ with a basis of $\mathrm{Ind}(D)$, whose induced coverage is just $fin$. Hence we have that
$$\mathrm{Sh}((D,fin), \mathbf{2}) \cong \mathrm{Ind}(D).$$
\end{example}

\section{Coherent spaces via finitary Grothendieck topologies}

Coherent spaces appear naturally in a variety of contexts that deal with finitary properties. They are also called spectral spaces, as they appear as spectra of commutative rings. A comprehensive account can be found for example in \cite{Dickmann_Schwartz_Tressl_2019}. Let us begin with a definition.

\begin{definition}
A frame $F$ is called \emph{coherent} if it is of the form $F = \mathrm{Ind}(D)$ for some bounded distributive lattice $D$. A locale $L$ is called \emph{coherent} if its frame of open sets is coherent.
\end{definition}

A sober space $X$ such that $\Omega(X)$ is a coherent locale is also called \emph{coherent} or \emph{spectral space}. We will give a variety of different characterizations in the following. Before we do so, let us introduce an important adjunction.

\begin{theorem}[Stone duality for coherent locales] \label{stonedualitycoherent}
There is an adjunction
\[\begin{tikzcd}[ampersand replacement=\&, column sep=small]
	{\mathrm{DLatt}_{bd}} \& {\mathrm{Frm}.}
	\arrow[""{name=0, anchor=center, inner sep=0}, "{\mathrm{Ind}}", curve={height=-12pt}, from=1-1, to=1-2]
	\arrow[""{name=1, anchor=center, inner sep=0}, "{\mathrm{forget}}", curve={height=-12pt}, from=1-2, to=1-1]
	\arrow["\dashv"{anchor=center, rotate=-90}, draw=none, from=0, to=1]
\end{tikzcd}\]
Moreover, $\mathrm{Ind}$ identifies ${\mathrm{DLatt}_{bd}}$ with the subcategory of $\mathrm{Frm}$ given by coherent frames and coherent frame homomorphisms.
\end{theorem}

Here, a frame homomorphism $f^* : F \rightarrow F'$ is called \emph{coherent} if it preserves compact objects. For a proof see \cite[p. 59]{johnstone1982stone}. We note that Johnstone requires all distributive lattices to be bounded.

Since points of a locale $L$ are given by frame homomorphisms $\Omega(L) \rightarrow \mathbf{2}$, by the adjunction we obtain for a bounded distributive lattice $D$ the isomorphism
$$\mathrm{pts}(\mathrm{Ind}(D)) = \mathrm{Hom}_{\mathrm{Frm}}( \mathrm{Ind}(D), \mathbf{2}) \cong \mathrm{Hom}_{\mathrm{DLatt}_{bd}}( D, \mathbf{2}),$$
hence it makes sense to refer to the set $\mathrm{Hom}_{\mathrm{DLatt}_{bd}}( D, \mathbf{2})$ as the set of points of $D$.

\begin{remark} The adjunction above gives $\mathrm{Ind}$ the structure of a comonad on the category $\mathrm{Frm}$, or dually the structure of a monad on $\mathrm{Loc}$. We can identify $\mathrm{DLatt}_{bd} \cong \mathrm{Ind}(\mathrm{DLatt}_{fin})$ where $\mathrm{DLatt}_{fin} \cong \mathrm{FinFrm}$ is the category of finite distributive lattices, equivalently finite frames. Under duality, this means we can interpret $\mathrm{Ind}$ as the codensity monad for the inclusion
$$\mathrm{FinLoc} \rightarrow \mathrm{Loc}$$
where $\mathrm{FinLoc}$ is the category of locales corresponding to finite frames. This codensity monad is a localic analog of the ultrafilter monad $\beta$ on the category of sets, whose algebras are compact Hausdorff spaces. \cite{leinster2013codensityultrafiltermonad}

We note that if $X$ is a locally compact, compact, and quasi-separated locale, it acquires the structure of an algebra for the monad $\mathrm{Ind}$. \cite[Section 3.3]{anel_lejay_exponentiable}
\end{remark}

\begin{theorem} Let $L$ be a locale. The following are equivalent.
\begin{enumerate}
\item $L$ is coherent.
\item $L$ is spatial and its corresponding topological space is compact, sober, and has a basis of compact open sets, which is closed under finite intersections.
\item $L$ is obtained as the spectrum of a commutative ring.
\item $L$ is obtained as an inverse limit of a diagram of finite posets equipped with the Alexandroff topology (in the category of locales).
\end{enumerate}
\end{theorem}

\begin{proof} The equivalence of $(1)$ and $(2)$ is treated in \cite[p. 65ff]{johnstone1982stone}. The equivalence of $(2)$ and $(3)$ is due to Hochster, and will not be relevant for the rest of this paper. We refer the reader to \cite[Section 12.6]{Dickmann_Schwartz_Tressl_2019}, in case of interest.

The equivalence of $(1)$ and $(4)$ is a consequence of the Stone duality provided above. Note that any bounded distributive lattice is given as $D \cong \colim_{i \in I} D_i$ with $D_i$ finite distributive lattices, and $I$ a filtered diagram (e.g. let $I$ be the set of finite sublattices of $D$). Since $\mathrm{Ind} : \mathrm{DLatt}_{bd} \rightarrow \mathrm{Frm}$ is a left adjoint, it sends this diagram to
$$ \mathrm{Ind}(D) \cong \mathrm{Ind}( \colim_{i \in I} D_i ) \cong \colim_{i \in I} \mathrm{Ind}(D_i) \cong \colim_{i \in I} D_i$$
with the colimit in the third and fourth terms taken in the category of frames. For the last isomorphism, we used that finite distributive lattices are already complete under filtered colimits. We will see in Section \ref{alexandroff} that finite frames correspond precisely to spaces obtained from equipping finite posets with the Alexandroff topology. Thus, passing to the category of locales, we thus see that the corresponding locale is obtained as an inverse limit of finite posets. Conversely, assume  $L$ is an inverse limit of finite posets, or equivalently its frame is given as a filtered colimit $F = \colim_{i \in I} D_i$ of finite frames. Note that a frame homomorphism between finite frames is automatically coherent, hence the entire diagram comes from a diagram of bounded distributive lattices, therefore, again using that $\mathrm{Ind}$ preserves colimits, the frame of $L$ is in the essential image of the functor $\mathrm{Ind}$. \end{proof}

\begin{remark}
Since the functor $\mathrm{pts}: \mathrm{Loc} \rightarrow \mathrm{Top}$ preserves limits, it follows that a coherent space $X$ is also an inverse limit of finite posets in $\mathrm{Top}$.
\end{remark}

\subsection{Locally coherent spaces}

We can loosen the requirement for $D$ to be a bounded distributive lattice to just be lower bounded, i.e. we do not require $D$ to have a top element. Geometrically, this corresponds to the removal of compactness from our requirements.

\begin{definition}
A frame $F$ is called \emph{locally coherent} if it is of the form $F = \mathrm{Ind}(D)$ for some lower bounded distributive lattice $D$. A locale $L$ is called \emph{locally coherent} if its frame of open sets is locally coherent.
\end{definition}

\begin{theorem}
Let $L$ be a locale. The following are equivalent.
\begin{enumerate}
\item $L$ is locally coherent.
\item The frame of $L$ is compactly generated, and its compact generators are closed under meets.
\item $L$ is is spatial and its corresponding topological space is sober and has a basis of compact open sets, which is closed under finite intersections.
\item $L$ is obtained from a coherent, locally cartesian $0$-site.
\end{enumerate}
\end{theorem}

\begin{proof}
The equivalence of $(1)$ and $(2)$ is mostly formal. Let $F$ be the frame of $L$. If $(1)$ holds and we have $F = \mathrm{Ind}(D)$ for some lower bounded distributive lattice $D$, we have that the set of compact opens $F^\omega \cong D$ is closed under meets. By definition, $D$ generates $F$. Conversely, given $(2)$, the sub-poset $F^\omega \subset F$ given by compact objects is always closed under finite colimits (= joins). The condition that it is closed under meets simply implies that $F^\omega$ is a sub-lattice of $F$, in particular lower bounded and distributive. Compact generation now says that $F = \mathrm{Ind}(F^\omega)$.

To see $(2)$ implies $(3)$, the only statement missing is that $L$ is spatial. This follows from the general fact that a locale with continuous frame is automatically spatial, see \cite{picado_pultr}[VII, 6.3.4], and a compactly generated frame is in particular continuous. The converse $(3)$ implies $(2)$ is clear.

To see $(1)$ implies $(4)$, as shown in Example \ref{indfinitetoplogy} we have that $\mathrm{Ind}(D) \cong \mathrm{Sh}((D,fin), \mathbf{2}),$ where $(D,fin)$ is a coherent, locally cartesian $0$-site. Conversely, for $(4)$ implies $(1)$, assume $(P,\tau)$ is a coherent, locally cartesian $0$-site. We have the adjunction

\[\begin{tikzcd}
	{\mathrm{Sh}((P,\tau), \mathbf{2})} & {\mathrm{Fun}(P^{op}, \mathbf{2}).}
	\arrow[""{name=0, anchor=center, inner sep=0}, curve={height=12pt}, hook', from=1-1, to=1-2]
	\arrow[""{name=1, anchor=center, inner sep=0}, "(-)^{sh}"', curve={height=12pt}, from=1-2, to=1-1]
	\arrow["\dashv"{anchor=center, rotate=-90}, draw=none, from=1, to=0]
\end{tikzcd}\]
Since the sheaf condition for a coherent locally cartesian $0$-site is a collection of finite limit conditions, the sub-poset of propositional sheaves is closed under filtered colimits (= directed suprema). This implies that $(-)^{sh}$ preserves compact objects. In particular, the compact generators given by the image of the Yoneda embedding $y : P \rightarrow \mathrm{Fun}(P^{op}, \mathbf{2})$ are sent to compact generators. Compact objects are always closed under finite colimits. Let $D \subset \mathrm{Sh}((P,\tau), \mathbf{2})$ be the closure under finite joins of images of representables. We have that $\mathrm{Ind}(D) \cong \mathrm{Sh}((P,\tau), \mathbf{2})$. We are left to argue that $D$ is closed under finite meets. Let $U_i , i = 1, \hdots, n$ and $V_j, j =  1, \hdots, m$ be finite collections of elements of $P$. Using that $(-)^{sh}$ preserves finite meets (see Lemma \ref{leftadjointgeometricmorphism}), we have that
$$ \left( \bigvee_{i = 1}^n ( y_{U_i})^{sh} \right) \wedge \left( \bigvee_{j = 1}^m ( y_{V_j} )^{sh} \right) = \bigvee_{i = 1}^n \bigvee_{j = 1}^m ( y_{U_i} )^{sh} \wedge ( y_{V_j} )^{sh} = \bigvee_{i = 1}^n \bigvee_{j = 1}^m ( y_{(U_i \wedge V_j)} )^{sh}.$$
\end{proof}

\begin{remark}
There is already a notion of coherent $\infty$-topos as defined by Lurie \cite[Definition A.2.0.12.]{lurieSAG}. Interpreting locales as $0$-topoi, which sit fully faithfully inside the $\infty$-category of $\infty$-topoi, it would be more consistent to call what we have defined as locally coherent locales simply ``coherent locales'', and specify compactness when one wants to talk about what we have defined in the previous section as a coherent locale (Or refer to them as spectral spaces). However, the term coherent space was already defined and used much earlier, e.g. in \cite{johnstone1982stone}, and the author felt inclined to avoid this clash with classical terminology.
\end{remark}

\begin{example} \label{polytopelattice}
Let $X$ be a classical euclidean, spherical or hyperbolic geometry, by which we mean that $X = E^n, S^n$ or $H^n$ for some fixed $n$. A geometric $n$-simplex is the convex hull of $n+1$-points of non-trivial measure (meaning the $n+1$-points do not lie on an $n-1$-dimensional subspace). A polytope in $X$ is a finite union of geometric $n$-simplices. Consider the poset $D(X)$ of $n$-dimensional polytopes, with morphisms being inclusions. This poset is a distributive lattice, but does not have a top element (however, we allow $\emptyset$ to be the bottom element).
\end{example}

We have seen in the previous section that the category $\mathrm{CohSp}$ is anti-equivalent to the category $\mathrm{DLatt}_{bd}$. Unfortunately, the relationship between locally coherent spaces and lower bounded distributive lattices is not so straightforward. This is because homomorphisms of lower bounded distributive lattices do not necessarily induce continuous maps. However, a useful trick to reduce statements about locally coherent spaces is the following. Consider the adjunction
\[\begin{tikzcd}
	{\mathrm{DLatt}_{lb}} & {\mathrm{DLatt}_{bd}}
	\arrow[""{name=0, anchor=center, inner sep=0}, "{(-)_\infty}", curve={height=-12pt}, from=1-1, to=1-2]
	\arrow[""{name=1, anchor=center, inner sep=0}, "forget", curve={height=-12pt}, from=1-2, to=1-1]
	\arrow["\dashv"{anchor=center, rotate=-90}, draw=none, from=0, to=1]
\end{tikzcd}\]
where $(-)_\infty$ is the functor that adds a top element $\infty$ to a given distributive lattice. Since for a given lower bounded distributive lattice $D$, there is always a bounded lattice homomorphism $D_\infty \rightarrow \mathbf{2} = \{0 \leq 1\}$, that sends every element $\neq \infty$ to $0$, we get by formal nonsense a lift of the adjunction:
\[\begin{tikzcd}
	{\mathrm{DLatt}_{lb}} & {{\mathrm{DLatt}_{bd}}_{/\mathbf{2}}} & {(\mathrm{CohSp}_{\mathrm{pt} / })^{op}}
	\arrow[""{name=0, anchor=center, inner sep=0}, "{(-)_\infty}", curve={height=-12pt}, hook, from=1-1, to=1-2]
	\arrow[""{name=1, anchor=center, inner sep=0}, "R", curve={height=-12pt}, from=1-2, to=1-1]
	\arrow["\cong", from=1-2, to=1-3]
	\arrow["\dashv"{anchor=center, rotate=-90}, draw=none, from=0, to=1]
\end{tikzcd}\]
where $R( p :  D \rightarrow \mathbf{2} ) = p^{-1}(0)$ selects the ideal of elements of $D$ that are sent to zero, and $(-)_\infty$ becomes fully faithful. This allows us to define the category of locally coherent spaces as the full subcategory
$$
\mathrm{LocCohSp} \subset \mathrm{CohSp}_{\mathrm{pt} / }.
$$
spanned by the image of $(-)_\infty$ and we have a corresponding Stone duality with lower bounded distributive lattices. However, we note that the resulting maps do not simply correspond to continuous maps, but rather partially defined coherent maps, with open support.

Given a lower bounded distributive lattice $D$, we then have an open-closed decomposition
\begin{equation} \mathrm{Ind}(D) \rightarrow \mathrm{Ind}(D_\infty) \rightarrow \mathbf{2}. \label{openclosedtopelement}
\end{equation}
which identifies $\mathrm{Ind}(D)$ with $\mathrm{Ind}(D_\infty)_{/U}$, where $U = \bigvee_{V \in D} V$.

\begin{remark}
If $L$ is a locally coherent locale, and $L = \mathrm{Ind}(D)$ with $D$ a lower bounded distributive lattice, then points of $L$ can be described purely in terms of $D$ as \emph{prime filters} on $D$. A filter $\mathcal{F} \subset D$ is called prime, if:
\begin{itemize}
\item For all $U, V \in D$ such that $U \vee V \in \mathcal{F}$, then $U \in \mathcal{F}$ or $V \in \mathcal{F}$.  
\end{itemize} 
\end{remark}

\subsection{Boolean algebras and profiniteness}

We now turn to the subclass of coherent spaces given by Stone spaces, also called profinite sets. Algebraically speaking, this corresponds to considering Boolean algebras instead of bounded distributive lattices. Since the algebraic theory of Boolean algebras is obtained from the theory of bounded distributive lattices by adding the single operation of negation together with additional axioms, we have an adjunction

\[\begin{tikzcd}
	{\mathrm{DLatt}_{bd}} & {\mathrm{BAlg}}
	\arrow[""{name=0, anchor=center, inner sep=0}, "{\mathrm{Bool}}", curve={height=-12pt}, from=1-1, to=1-2]
	\arrow[""{name=1, anchor=center, inner sep=0}, "\mathrm{forget}", curve={height=-12pt}, from=1-2, to=1-1]
	\arrow["\dashv"{anchor=center, rotate=-90}, draw=none, from=0, to=1]
\end{tikzcd}\]

Since it is only a property for a distributive lattice to be a Boolean algebra, the right adjoint forgetful functor is in fact fully faithful.	The distributive lattice $\mathbf{2}$ is Boolean, therefore we have for $D$ a bounded distributive lattice a natural isomorphism
$$ \mathrm{pts}(D) = \mathrm{Hom}_{\mathrm{DLatt}_{bd}}(D, \mathbf{2}) \cong \mathrm{Hom}_{\mathrm{DLatt}_{bd}}(\mathrm{Bool}(D), \mathbf{2}) \cong \mathrm{pts}(\mathrm{Bool}(D)).$$
Thinking topologically, if $X$ is a coherent space with set of compact opens given by $D$, this means that the coherent space corresponding to $\mathrm{Bool}(D)$ is given by equipping the space $X$ with a new topology. We call this the \emph{constructible topology} and write $X^{cons}$. Let us summarize some known results:

\begin{proposition}[See \cite{stacks-project}, 5.23]
Let $L$ be a coherent locale. The following are equivalent:
\begin{enumerate}
\item $L$ is obtained as an inverse limit of finite sets with the discrete topology (in the category of locales).
\item The space of points of $L$ is Hausdorff.
\item The space of points of $L$ is totally disconnected.
\item Every compact open of $L$ has a complement.
\item The frame of $L$ is given as $\mathrm{Ind}(D)$, where $D$ is a boolean algebra.
\item There are no non-trivial specializations between points.
\item The constructible topology on the space of points of $L$ equals the given topology.
\end{enumerate}
\end{proposition}

We refer to a locale with any of the above properties as a Stone locale, or also as a profinite space. We will use the following elementary fact about profinite spaces.

\begin{lemma} \label{profinitecontinuous}
Let $X$ be a profinite space given as a cofiltered limit $X = \lim_{i \in I} X_i$, with $X_i$ being finite, discrete sets.\footnote{The limit can be taken either in the category of topological spaces or locales, the distinction does not matter here.} Let $S$ be an (arbitrary) discrete set. Then
$$ \mathrm{Map}( X, S ) \cong \mathrm{colim}_{i \in I} \mathrm{Map}( X_i, S ).$$
\end{lemma}

\begin{proof}
A map $f$ from $X$ into $S$ is continuous iff it is locally constant, since $S$ is discrete. As such for any given $f$ there is a partition of $X$ into closed and open subsets $U_j, j \in J$, such that $f|_{U_j}$ is a constant function. Since $X$ is compact, this partition can be chosen to be finite. Furthermore, since $X$ is Hausdorff, an open subset is closed iff it is compact. Hence for $f$ we have a finite partition into compact open subsets $U_j, j \in J$, on which $f$ is constant.

Now note that under Stone duality the Boolean algebra of compact open subsets of $X$ is obtained as the filtered colimit of the Boolean algebras of compact open subsets of $X_i$. Hence for our finite (!) collection $U_j, j \in J$ we can find a given index $i_0 \in I$ such that all $U_j$ are obtained as preimages of compact open subsets under the structure map $X \rightarrow X_{i_0}$. But this just means that $f$ factors through $X \rightarrow X_{i_0}$.
\end{proof}

We will also make use of the following non-trivial fact about continuous functions with values in the integers $\mathbb{Z}$ on a profinite space.

\begin{theorem}[Nöbeling's Theorem, see \cite{asgeirsson:LIPIcs.ITP.2024.6}] \label{noebelingstheorem}
Let $X$ be a profinite space. Then $C(X, \mathbb{Z})$ is a free abelian group.
\end{theorem}

\begin{example}
A useful example of a profinite space is the Stone-\v{C}ech compactification $\beta(S)$ of a set $S$. Consider a finite partition $\Pi = \{U_1, U_2, \hdots, U_n\}$ of $S$, by which we mean that
$$S = \bigcup_{i = 1}^n U_i$$
and the $U_i$ are pairwise disjoint. We have a natural map $S \rightarrow \Pi$ that sends an element $s \in S$ to the unique set $U_i$ that contains $s$. Conversely, any map $S \rightarrow F$ with $F$ a finite set induces a finite partition on $S$. Note that the set of finite partitions of $S$ is ordered: We can say $\Pi \leq \Pi'$ if every element $U$ of $\Pi$ is obtained as a union of elements in $\Pi'$.

We can form the colimit and observe that we have an isomorphism
$$\mathcal{P}(S) \cong \mathrm{colim}_{ \Pi \text{ finite partition of } S} \mathcal{P}(\Pi),$$
in the category of Boolean algebras, where $\mathcal{P}(S)$ is the powerset algebra of the set $S$ (and similarly for $\Pi$). This isomorphism sends a subset $U \subset S$ to the singleton $\{U\}$ of the partition $\{U,U^c\}$ of $S$ and its inverse takes the union of the elements of a subset of a partition. Define $\beta(S)$ to be the profinite space associated to $\mathcal{P}(S)$. Under Stone-duality, this translates to the homeomorphism
$$\beta(S) \cong \lim_{ \Pi \text{ finite partition of } S} \Pi.$$
with the frame of $\beta(S)$ obtained as $\mathrm{Ind}(\mathcal{P}(S))$.

The identity on $\mathcal{P}(S)$ extends to a frame homomorphism $\mathrm{Ind}(\mathcal{P}(S)) \rightarrow \mathcal{P}(S)$, or dually a continuous map $S^{disc} \rightarrow \beta(S)$, where we equip $S$ with the discrete topology. We leave it the reader to verify that this map is an open, dense inclusion. Furthermore observe that the space $\beta(S)$ can be obtained from the finite disjoint covering topology on $\mathcal{P}(S)$, more or less by definition.
\end{example}

\subsection{Generalities on sheaves and higher topoi} \label{generalitiestopoi}

We will use the language of higher topoi and categories of sheaves with values in the $\infty$-category of spaces as developed by Lurie, \cite{luriehtt}. Given an $\infty$-category $\mathcal{C}$ we refer to $\mathrm{PSh}(\mathcal{C}) = \mathrm{Fun}(\mathcal{C}^{op}, \mathrm{Spc})$ as the $\infty$-category of presheaves. If $\mathcal{C}$ is equipped with a Grothendieck topology $\tau$ we refer to $\mathrm{Sh}(\mathcal{C}, \tau) \subset \mathrm{PSh}(C)$ as the full subcategory of sheaves with respect to $\tau$. In case of a locale or space $X$ we simply write $\mathrm{Sh}(X)$ for sheaves with respect to the canonical topology on the corresponding frame. 

Given a locale $L$ and an open $U$, we have adjunctions
\[\begin{tikzcd}
	{\mathrm{Sh}(L_{U/})} & {\mathrm{Sh}(L)_{/y_U}} & {\mathrm{Sh}(L)} & {\mathrm{Sh}(L_{U/})}
	\arrow["\simeq"{description}, draw=none, from=1-1, to=1-2]
	\arrow[""{name=0, anchor=center, inner sep=0}, "{i_!}", curve={height=-12pt}, hook, from=1-2, to=1-3]
	\arrow[""{name=1, anchor=center, inner sep=0}, "{i_*}"', curve={height=12pt}, hook, from=1-2, to=1-3]
	\arrow[""{name=2, anchor=center, inner sep=0}, "{i^*}"{description}, from=1-3, to=1-2]
	\arrow[""{name=3, anchor=center, inner sep=0}, "{c^*}", curve={height=-12pt}, from=1-3, to=1-4]
	\arrow[""{name=4, anchor=center, inner sep=0}, "{c_*}", hook', from=1-4, to=1-3]
	\arrow["\dashv"{anchor=center, rotate=-90}, draw=none, from=0, to=2]
	\arrow["\dashv"{anchor=center, rotate=-90}, draw=none, from=2, to=1]
	\arrow["\dashv"{anchor=center, rotate=-90}, draw=none, from=3, to=4]
\end{tikzcd}\]
corresponding to the open-closed decomposition of $L$ given by $U$ and its closed complement. The functors $i_!, i_*$ and $c^*$ are fully faithful and $c_*$ preserves filtered colimits, see \cite[Section 6.3.5 and Section 7.3.2]{luriehtt}.

Particularly useful will be the so-called comparison lemma, which is an $\infty$-categorical generalization of the Basis Theorem \ref{basistheorem}.

\begin{lemma}[Comparison lemma, see \cite{hoyois_2014}, Lemma C.3] \label{comparisonlemma}
Let $(P, \tau)$ be a locally cartesian $0$-site and $u : P_0 \subset P$ a subset of $P$ such that:
\begin{itemize}
\item Every object in $P$ can be covered by objects in $P_0$.
\item $P_0$ is closed under meets in $P$.
\end{itemize}
Let $\tau_0$ be the induced Grothendieck topology on $P_0$ by restriction. Then the induced adjunction
\[\begin{tikzcd}
	{\mathrm{Sh}(P, \tau)} & {\mathrm{Sh}(P_0, \tau_0)}
	\arrow[""{name=0, anchor=center, inner sep=0}, "{u^*}", curve={height=-12pt}, from=1-1, to=1-2]
	\arrow[""{name=1, anchor=center, inner sep=0}, "{u_*}", curve={height=-12pt}, from=1-2, to=1-1]
	\arrow["\dashv"{anchor=center, rotate=-90}, draw=none, from=0, to=1]
\end{tikzcd}\]
is an equivalence of $\infty$-categories of sheaves.
\end{lemma}

We note that Lemma \ref{comparisonlemma} as presented is a very special case of the comparison lemma discussed by Hoyois \cite{hoyois_2014}, however it will be the only case we need during this paper. As a warning to the reader: In the case of $1$-topoi a generalization of Lemma \ref{comparisonlemma} avoiding the requirement of closure under meets is possible, however, when dealing with the corresponding $\infty$-topoi this runs into issues related to hypercompletion, see e.g. \cite{dyckerhoff2025hypersheavesbases}.

\begin{proposition}
Let $X$ be a locally coherent space, represented by the frame $\mathrm{Ind}(D)$ for a lower bounded distributive lattice $D$. Then there is an equivalence
$$\mathrm{Sh}(X) \cong \mathrm{Sh}(D,fin)$$
where on the right-hand side $D$ is equipped with the finite covering topology.
\end{proposition}

\begin{proof}
This follows immediately from Lemma \ref{comparisonlemma} by applying it to the inclusion $D \rightarrow \mathrm{Ind}(D)$. The induced topology on $D$ is simply the finite covering topology.
\end{proof}

We note that the sheaf condition for the finite covering topology reduces to the following two conditions. A presheaf $\mathcal{F} : D^{op} \rightarrow \mathrm{Spc}$ is a sheaf for the finite covering topology iff:
\begin{itemize}
\item $\mathcal{F}( 0 ) \simeq 1$
\item For all $U, V \in D$ the square
\[\begin{tikzcd}
	{\mathcal{F}(U \vee V)} & {\mathcal{F}(U)} \\
	{\mathcal{F}(V)} & {\mathcal{F}(U \wedge V)}
	\arrow[from=1-1, to=1-2]
	\arrow[from=1-1, to=2-1]
	\arrow[from=1-2, to=2-2]
	\arrow[from=2-1, to=2-2]
\end{tikzcd}\]
is a pullback.
\end{itemize}
In the case of coherent spaces, this description already appeared in \cite{luriehtt} as Theorem 7.3.5.2.

\begin{corollary} \label{compgen}
Let $X$ be a locally coherent space. Then $\mathrm{Sh}(X)$ is compactly generated.
\end{corollary}

\begin{proof}
By the previous proposition, we have $\mathrm{Sh}(X) \cong \mathrm{Sh}(D,fin)$. Consider the adjunction
\[\begin{tikzcd}
	{\mathrm{Sh}(D,fin)} & {\mathrm{PSh}(D)}
	\arrow[""{name=0, anchor=center, inner sep=0}, curve={height=12pt}, hook, from=1-1, to=1-2]
	\arrow[""{name=1, anchor=center, inner sep=0}, "(-)^{sh}"', curve={height=12pt}, from=1-2, to=1-1]
	\arrow["\dashv"{anchor=center, rotate=-90}, draw=none, from=1, to=0]
\end{tikzcd}\]
Since $\mathrm{Sh}(D,fin)$ is generated under colimits by the images of representables, it suffices to argue that $(-)^{sh}$ preserves compact object, which would follow from the claim that the inclusion of sheaves into presheaves preserves filtered colimits, \cite[Lemma 5.5.1.4.]{luriehtt}. This claim is true as the sheaf condition for the finite covering topology is given by a collection of finite limit conditions, which are stable under filtered colimits.
\end{proof}

\begin{lemma} \label{compmor}
Let $f : D \rightarrow D'$ be a homomorphism of lower bounded distributive lattices. Then the induced left adjoint functor
$$f^* : \mathrm{Sh}(D,fin) \rightarrow \mathrm{Sh}(D',fin)$$
preserves compact objects.
\end{lemma}

\begin{proof}
This is immediate as $f^*$ sends representables to representables by construction. These are compact generators, and $f^*$ preserves colimits (in particular finite colimits) and retracts. 
\end{proof}

We will make use of some properties of the $\infty$-category of $\infty$-topoi. We refer to $\mathrm{RTop}$ as the (large) $\infty$-category of $\infty$-topoi and right adjoints $f_*$ of geometric morphisms between them. The following is a special case of Lurie \cite{luriehtt}, Proposition 6.4.5.7. and Definition 6.4.5.8.
\begin{theorem} \label{adjunctionlocalestopoi}
There is an adjunction
\[\begin{tikzcd}
	{\mathrm{Loc}} & {\mathrm{RTop}}
	\arrow[""{name=0, anchor=center, inner sep=0}, "{\mathrm{Sh}(-)}"', curve={height=12pt}, hook, from=1-1, to=1-2]
	\arrow[""{name=1, anchor=center, inner sep=0}, "{\mathrm{Sub}(1)}"', curve={height=12pt}, from=1-2, to=1-1]
	\arrow["\dashv"{anchor=center, rotate=-90}, draw=none, from=1, to=0]
\end{tikzcd}\]
\end{theorem}

We will also use the following statement about cofiltered limits in the $\infty$-category $\mathrm{RTop}$. Here, $\widehat{\mathrm{Cat}_\infty}$ refers to the (very large) $\infty$-category of large $\infty$-categories.

\begin{theorem}[See \cite{luriehtt}, Theorem 6.3.3.1.] \label{cofilteredlimitstopoi}
The category $\mathrm{RTop}$ of $\infty$-topoi and right adjoints of geometric morphisms between them has all cofiltered limits, and the forget functor $\mathrm{RTop} \rightarrow \widehat{\mathrm{Cat}_\infty}$ preserves them.
\end{theorem}

\subsection{Alexandroff spaces and Birkhoff's theorem} \label{alexandroff}

Let $P$ be a poset. We can equip $P$ with the Alexandroff topology, where the open sets are given by lower closed subsets of $P$. Let us write $P_{Alex}$ for the resulting topological space. The corresponding frame of opens is equivalently given as $\mathrm{Fun}(P^{op},\mathbf{2})$ and the natural monotone map that associates to $p \in P$ the set $p\hspace{-0.5ex}\downarrow ~= \{ q \in P ~|~ q \leq p \}$ corresponds to the $0$-categorical Yoneda embedding
$$y : P \rightarrow  \mathrm{Fun}(P^{op},\mathbf{2}).$$
In the following, given an $\infty$-topos $\mathcal{X}$, we denote its hypercompletion by $\mathcal{X}^{hyp}$, see \cite[Section 6.5.2]{luriehtt}.

\begin{proposition}[\cite{Aoki_2023}, Example A.11.]
Let $P$ be a poset. Then there is a natural equivalence of $\infty$-categories
$$ \mathrm{PSh}(P) \simeq \mathrm{Sh}( P_{Alex} )^{hyp}$$
given by right Kan extending the composition
$$ P \xrightarrow{y} \mathrm{Fun}(P^{op},\mathbf{2}) \rightarrow \mathrm{Sh}( P_{Alex} ).$$
\end{proposition}

We need the following special case if the poset $P$ is \emph{finite}.

\begin{corollary} \label{finiteposet}
Let $P$ be a \emph{finite} poset. Then $\mathrm{Sh}( P_{Alex} )$ is hypercomplete. In particular, we have a natural equivalence
$$ \mathrm{PSh}(P) \simeq \mathrm{Sh}( P_{Alex} ).$$
\end{corollary}

\begin{proof}
The proof of this reduces to two facts.
\begin{enumerate}
\item \cite{luriehtt}, Corollary 7.2.1.12. An $\infty$-topos $\mathcal{X}$ which is homotopy dimension $\leq n$ for some $n$ is hypercomplete.
\item \cite{luriehtt}, Theorem 7.2.3.6. Let $X$ be a paracompact topological space of covering dimension $\leq n$. Then the $\infty$-topos $\mathrm{Shv}(X)$ has homotopy dimension $\leq n$.
\end{enumerate}
It is clear that the latter conditions are satisfied for a finite poset.
\end{proof}

Given a \emph{finite} poset $(P,\leq)$, its frame of opens with respect to the Alexandroff topology, $\mathrm{Fun}(P^{op},\mathbf{2})$ is a finite distributive lattice, or equivalently a finite frame (the existence of all colimits and limits in this case is automatic).

Conversely, given a finite frame $D$, we can associate to it its poset of points $\mathrm{pts}(D)$ together with the specialization order. Recall that in a general frame $F$, a point is given as a frame homomorphism $F \rightarrow \mathbf{2}$, or equivalently as a completely prime filter $\mathcal{F}$ on $F$. In a finite frame, any such prime filter $\mathcal{F}$ must have a minimal element $p$, which is an irreducible element in $D$, i.e. an element $U \in D$ such that whenever $U = V_1 \vee V_2$, then $U \leq V_1$ or $U \leq V_2$. Hence we can identify the poset of points with a sub-poset $\mathrm{pts}(D) \rightarrow D$ spanned by the irreducible elements. 

\begin{theorem}[Birkhoff, see \cite{grätzer2011lattice} II 1.3]
There is an equivalence of categories
$$\mathrm{pts} : \mathrm{DLatt}_{fin} \simeq \mathrm{FinPoset}^{op} : \mathcal{O}$$
given by sending a finite distributive lattice to its poset of points, and by sending a finite poset to the frame obtained from the Alexandroff topology. 
\end{theorem}

Under this correspondence, the category of sheaves on a finite frame has a particularly simple description, which follows directly from Corollary \ref{finiteposet}.

\begin{proposition} \label{sheavesonfiniteframe}
Suppose $D$ is a finite frame. Then there is a natural equivalence
$$\mathrm{Sh}(D) \simeq \mathrm{Fun}( \mathrm{pts}(D)^{op}, \mathrm{Spc} )$$
obtained by restricting a sheaf $G : D \rightarrow \mathrm{Spc}$ along the inclusion $\mathrm{pts}(D) \rightarrow D$.
\end{proposition}

We will need a last remark about the Booleanization of a finite frame.

\begin{proposition} \label{finiteboolean}
Let $D$ be a finite frame. Then $$\mathrm{Bool}(D) \cong \mathcal{P}( \mathrm{pts}(D) )$$
is the set of subsets of the set of points of $D$.
\end{proposition}

\begin{proof}
We can write $D = \Omega(P_{\mathrm{Alex}})$ where $P = \mathrm{pts}(D)$. Since $\mathrm{Bool}(D)$ is given as a topology on $P$ as a set, which contains $\Omega(P_{\mathrm{Alex}})$, we only need to show that all singletons $\{p\}$ for $p \in P$ lie in $\mathrm{Bool}(D)$.

Define the \emph{height function} $ht : P \rightarrow \mathbb{N}$ by induction. The height of a minimal element is defined to be $0$. Then the subset of height $n$ elements of $P$ is defined as the set of elements of $p \in P$ such that there exists an element $q \leq p$ of height $n-1$ and there does not exist another element $q'$ such that $p \leq q' \leq q$. It is clear that $ht$ is a strictly increasing map.

Suppose $p$ is an element of height $n$ and let $p \hspace{-0.5ex}\downarrow ~= \{ q \in P ~|~ q \leq p \}$ be the basic open subset of $P_{\mathrm{Alex}}$ associated to $p$. Then
$$p \hspace{-0.5ex}\downarrow ~= \{p\} \cup \{ q \in P ~|~ q \lneq p \}$$
decomposes into two disjoints set, where the second set only contains elements of height $\leq \{n-1\}$. The minimal elements of $P$ give basic opens, hence $\{p\} \in \mathrm{Bool}(D)$ is given. Now proceed by induction on the height to see that $\mathrm{Bool}(D)$ contains all singletons.
\end{proof}

\subsection{Valuations and motives for distributive lattices}
\label{sectionvaluations}

\begin{definition}
Let $D$ be a lower bounded, distributive lattice and $A$ an abelian group. An $A$-valued valuation on $D$ is a function $\mu : D \rightarrow A$ such that
\begin{enumerate}
\item $\mu(0) = 0$
\item \emph{Modularity:} For all $U, V \in D$ it holds that $\mu(U) + \mu(V) = \mu(U \vee V) + \mu(U \wedge V)$.
\end{enumerate}
We denote the set of $A$-valued valuations on $D$ by $\mathrm{Val}(D;A)$.
\end{definition}

For a fixed abelian group $A$, the assignment of $D$ to $\mathrm{Val}(D;A)$ is a contravariant functor. If $f : D \rightarrow D'$ is a homomorphism of lower bounded distributive lattices and $\mu : D' \rightarrow A$ is a valuation, then the map $f^* \mu = \mu \circ  f  : D \rightarrow A$ is again a valuation. This defines a map $f^* : \mathrm{Val}(D';A) \rightarrow \mathrm{Val}(D;A)$. The expression $\mathrm{Val}(D;A)$ is also clearly covariantly functorial in the abelian group $A$.

Given a lower bounded distributive lattice $D$, there exists a universal valuation $\mu_{univ} : D \rightarrow M(D)$, where $M(D)$ is the free abelian group $\mathbb{Z}[D]$ modulo the relations:
\begin{enumerate}
\item $[0] = 0$
\item $[U] + [V] = [U \vee V] + [U \wedge V]$ for all $U, V \in D$.
\end{enumerate}
It is universal in the sense that whenever $\mu : D \rightarrow A$ is a valuation, there exists a unique group homomorphism $M(D) \rightarrow A$ extending $\mu$ along $\mu_{univ}$. We call $M(D)$ the \emph{module of D-motives}. We summarize this in the following statement.

\begin{proposition}
The map $\mu_{univ} : D \rightarrow M(D)$ induces for each abelian group $A$ a natural bijection
$$\mathrm{Val}(D; A) \cong \mathrm{Hom}_{\mathrm{Ab}}( M(D); A ).$$
\end{proposition}

In particular, by the Yoneda embedding, $M : \mathrm{DLatt}_{lb} \rightarrow \mathrm{Ab}$ is a functor, a fact that is also clear from the presentation.

\begin{example}
Let $B$ be a Boolean algebra. Then $\mu : B \rightarrow A$ is a valuation iff it is a \emph{finitely additive} function, i.e. $\mu(0) = 0$ and $\mu(U \vee V) = \mu(U) + \mu(V)$ whenever $U,V$ are disjoint, by which we mean $U \wedge V = 0$.
\end{example}

\begin{example} \label{finiteset}
Let $S$ be a finite set and let $D = \mathcal{P}(S)$ be the boolean algebra of subsets of $S$. Then it is clear that a valuation $\mu$ is determined by its one-element sets $\{s\}$ for $s \in S$, as any other set is a finite disjoint union of those. Hence we have that
$$M(\mathcal{P}(S)) \cong \mathbb{Z}[S]$$
is the free abelian group with basis $S$.
\end{example}

\begin{example}
Let $X$ be a classical geometry, i.e. $X = E^n, S^n$ or $H^n$ (euclidean, spherical or hyperbolic). As discussed in Example \ref{polytopelattice} we have the lower bounded distributive lattice $D(X)$ of $n$-dimensional polytopes. The abelian group
$M(D(X))$ is also referred to as the \emph{polytope module} $\mathrm{Pt}(X)$ of $X$ in \cite{malkiewich2024higherscissorscongruence} and of central importance in scissors congruence $K$-theory, as it controls the homology of the scissors congruence $K$-theory spectrum of $X$, see e.g. \cite[Theorem 1.5]{malkiewich2024higherscissorscongruence}.
\end{example}

\begin{example}
Assume $X$ is a profinite space with $B =  \mathcal{K}^o(X)$ being the Boolean algebra of compact open subsets of $X$. Note that for profinite $X$ any compact open $U$ has a complement $U^c$ such that $X = U \amalg U^c$. Hence we can define the indicator function $1_U : X \rightarrow \mathbb{Z}$ as the function with value $1$ on $U$ and $0$ on the complement. This gives an assignment
$$\begin{array}{rcl}
\mu : B & \rightarrow & C(X; \mathbb{Z}) \\
U & \mapsto & 1_U.
\end{array}$$
It is easy to check that this is in fact a valuation on $B$, hence by universality we obtain a group homomorphism
$$M(B) \rightarrow C(X; \mathbb{Z}).$$
We will see that this is an isomorphism in Proposition \ref{motivesforboolean}.
\end{example}

\begin{example} \label{measurespacemotive} Let $L$ be a \emph{measurable locale} in the sense of \cite{Pavlov_2022}. For example $L$ could be obtained by considering a compact, $\sigma$-finite measure space $(X, \mathcal{L}, \mu)$ and let $\Omega(L)$ be the complete $\sigma$-algebra $\mathcal{L}/\mathcal{N}$, where $\mathcal{N}$ is the $\sigma$-ideal of $\mu$-null sets. Then $L^\infty(L)$ is defined as the set of bounded, complex valued continuous functions on $L$. (And agrees with $L^\infty(X, \mathcal{L}, \mu)$ for the given special case.)

Since $L$ is Boolean, any open $U$ gives a decomposition $L \cong L_{/U} \amalg L_{/U^{c}}$, thus allowing us to define the indicator function $\chi_U : L \rightarrow \mathbb{C}$ as the (continuous!) function with value $1$ on $U$ and $0$ on $U^c$.

One can then see that $M(\Omega(L))$ can be identified with the subset of $L^\infty(L)$ given by \emph{integral-valued step functions}, that is $\mathbb{Z}$-linear combinations of indicator functions. We see that $L^\infty(L)$ is in some sense a completed, $\mathbb{C}$-valued version of $M(\Omega(L))$.
\end{example}

\begin{lemma} \label{motivesfiltered}
The functor $M : \mathrm{DLatt}_{lb} \rightarrow \mathrm{Ab}$ preserves filtered colimits.
\end{lemma}

\begin{proof}
This is immediate from the presentation of the abelian group $M(D)$ as a quotient of $\mathbb{Z}[D]$, together with the observation that filtered colimits in $ \mathrm{DLatt}_{lb}$ are computed as filtered colimits of underlying sets.
\end{proof}

\begin{proposition} \label{motivesforboolean}
Let $X$ be a profinite space, corresponding to the Boolean algebra $B = \mathcal{K}^o(X)$. Then the assignment of each compact open $U$ to its indicator function induces an isomorphism
$$M(B) \cong C(X, \mathbb{Z}),$$
where the right hand side is the set of continuous functions of $X$ into $\mathbb{Z}$ equipped with the discrete topology. In particular, $M(B)$ is a free abelian group.
\end{proposition}

\begin{proof}
Let $X = \lim_{i \in I} X_i$ be a cofiltered limit with $X_i$ finite, discrete sets. This corresponds under Stone duality to
$$B = \colim_{i \in I} \mathcal{P}(X_i)$$
in the category of Boolean algebras. Using Lemma \ref{motivesfiltered} we have
$$M(B) \cong \colim_{i \in I} M(\mathcal{P}(X_i)) \cong \colim_{i \in I} \mathbb{Z}[X_i] \cong \colim_{i \in I} C(X_i, \mathbb{Z}) \cong C(X, \mathbb{Z})$$
where the last isomorphism is provided by Lemma \ref{profinitecontinuous}. The statement about freeness of $M(B)$ follows from Nöbeling’s Theorem, see Theorem \ref{noebelingstheorem}.
\end{proof}

\begin{proposition} \label{motivesonbooleanization}
Let $D$ be a bounded distributive lattice. The natural homomorphism
$D \rightarrow \mathrm{Bool}(D)$ induces an isomorphism
$$M(D) \cong M(\mathrm{Bool}(D)).$$
\end{proposition}

\begin{proof}
By Lemma \ref{motivesfiltered} it suffices to prove the claim for a finite frame $D$, in which case we can write $D = \Omega(P_{\mathrm{Alex}})$ for some finite poset $(P,\leq)$.

By the universal property of the module of motives $M(D)$, the claim is equivalent to showing that any valuation $\mu : D \rightarrow A$ there exists a unique extension $\mu_{\mathrm{Bool}} : \mathrm{Bool}(D) \rightarrow A$ along the homomorphism $D \rightarrow \mathrm{Bool}(D)$. By Proposition \ref{finiteboolean} and Example \ref{finiteset} we have
$$M( \mathrm{Bool}(D) ) \cong M( \mathcal{P}(P)) \cong \mathbb{Z}[P].$$
This means we are left to show that a valuation $\mu: \Omega(P_{\mathrm{Alex}}) \rightarrow A$ uniquely determines the values on the (not necessarily open) singleton sets $\{p\}$ for $p \in P$. We can do this by induction.

Take the height function $ht : P \rightarrow \mathbb{N}$ defined in the proof of Proposition \ref{finiteboolean}. Suppose $p$ is an element of height $n$ and let $p \downarrow = \{ q \in P ~|~ q \leq p \}$ be the basic open subset of $P_{\mathrm{Alex}}$ associated to $p$. Then
$$ p \downarrow = \{p\} \cup \{ q \in P ~|~ q \lneq p \}$$
decomposes into two disjoints set, where the second set only contains elements of height $\leq \{n-1\}$. Now assume $\mu: \Omega(P_{\mathrm{Alex}}) \rightarrow A$  is a valuation. The minimal elements of $P$ give basic opens, hence $\mu(\{p\})$ is given. Now proceed by induction on the height to see that $\mu$ is completely determined on all singletons.
\end{proof}

\begin{lemma} \label{splitofmot}
Let $D$ be a lower bounded distributive lattice. Then
$$M(D_{\infty}) \cong M(D) \oplus \mathbb{Z},$$
where the split is induced by the natural homomorphisms $D \rightarrow D_\infty \rightarrow \mathbf{2}$ and $\mathbf{2} \rightarrow D_\infty$.
\end{lemma}

\begin{proof}
It is clear that the maps $\mathbf{2} \rightarrow D_\infty$ and $D_\infty \rightarrow \mathbf{2}$ establish $\mathbf{2}$ as a retract of $D_\infty$ in the category $\mathrm{DLatt}_{lb}$. The identification of the resulting kernel with $M(D)$ is also straightforward.
\end{proof}

\begin{corollary} \label{motivesisfree}
Let $D$ be a lower bounded distributive lattice. Then $M(D)$ is a free abelian group.
\end{corollary}

\begin{proof}
Lemma \ref{splitofmot} reduces the statement from lower bounded distributive lattices to bounded distributive lattices, Proposition \ref{motivesonbooleanization} further reduces the statement to Boolean algebras and this case is covered by Proposition \ref{motivesforboolean}.
\end{proof}

The module of motives $M(D)$ can also be characterized via another universal property. Observe that $M(D)$ has a (well-defined and non-unital) commutative ring structure given by $[U] \cdot [V] = [U \wedge V]$. If $D$ has a top element as well, then $M(D)$ is a commutative ring with unit given by $[1]$. Homomorphisms in $\mathrm{DLatt}_{lb}$, i.e. homomorphisms of lattices that preserve the bottom element, evidently induce ring homomorphisms. Conversely, given any non-unital commutative ring $R$, we observe that the set of idempotent elements
$$ \mathrm{Idem}(R) = \{ p \in R ~|~ p^2 = p \} $$
obtains the structure of a lower bounded distributive lattice, with the operations
\begin{itemize}
\item $0$ is the bottom element,
\item $p \wedge q = p q$,
\item $p \vee q = p + q - pq$.
\end{itemize}
These two functors determine each other. Let $\mathrm{CAlg}^{nu}$ be the category of non-unital commutative rings.
\begin{theorem} \label{idempotents}
There exists an adjunction
\[\begin{tikzcd}
	{\mathrm{DLatt}_{lb}} & {\mathrm{CAlg}^{nu}}
	\arrow[""{name=0, anchor=center, inner sep=0}, "{M}", curve={height=-12pt}, from=1-1, to=1-2]
	\arrow[""{name=1, anchor=center, inner sep=0}, "{\mathrm{Idem}}", curve={height=-12pt}, from=1-2, to=1-1]
	\arrow["\dashv"{anchor=center, rotate=-90}, draw=none, from=0, to=1]
\end{tikzcd}\]
\end{theorem}
While this theorem can be checked directly, this adjunction will arise as a direct corollary of the adjunction constructed later in Theorem \ref{spaceofidempotents}. On a philosophical side, one could say that since the universal valuation of a (lower bounded) distributive lattice $D$ is determined by the concept of idempotent elements, this means that the concept of valuations arises as well naturally from that of idempotents.

\section{Presentable, compactly generated and dualizable $\infty$-categories}

We will denote the (very large) $\infty$-category of presentable $\infty$-categories and left adjoint functors by $\mathrm{Pr}^L$. We write  $\mathrm{Pr}^L_{ca}$ for the (non-full) subcategory of compactly assembled presentable $\infty$-categories and left adjoint functors, where corresponding right adjoints preserve filtered colimits, and $\mathrm{Pr}^L_{\omega}$ for the (full) subcategory of $\mathrm{Pr}^L_{ca}$ spanned by the compactly generated presentable $\infty$-categories. For resources discussing compactly assembled $\infty$-categories see \cite{krause_nikolaus_puetzstueck} and \cite{ramzi2024dualizablepresentableinftycategories}. We cite the following statements.

\begin{proposition}[See \cite{luriehtt} Theorem 5.5.3.18.] \label{collectionofcolimitspresentable}
The $\infty$-category $\mathrm{Pr}^L$ has all colimits and the composite
$$\mathrm{Pr}^L \simeq (\mathrm{Pr}^R)^{op} \xrightarrow{forget} (\widehat{\mathrm{Cat}}_\infty)^{op} $$
preserves and creates colimits.
\end{proposition}

\begin{proposition}[See \cite{luriehtt} Proposition 5.5.7.6.]
The $\infty$-category $\mathrm{Pr}^L_{\omega}$ has all colimits and the forget functor
$$\mathrm{Pr}^L_\omega \rightarrow \mathrm{Pr}^L$$
preserves and creates colimits.
\end{proposition}

\begin{proposition}[\cite{ramzi2024dualizablepresentableinftycategories} Proposition 2.49.]
The $\infty$-category $\mathrm{Pr}^L_{ca}$ has all colimits and the forget functor
$$\mathrm{Pr}^L_{ca} \rightarrow \mathrm{Pr}^L$$
preserves and creates colimits.
\end{proposition}

We note that in particular, the inclusion $\mathrm{Pr}^L_{\omega} \rightarrow \mathrm{Pr}^L_{ca}$ preserves all colimits (in fact, it has a right adjoint given by $\mathrm{Ind}((-)^{\omega})$.

By Theorem \ref{cofilteredlimitstopoi} we know that a cofiltered limit of $\infty$-topoi is computed as the limit of the underlying large categories along the diagram of right adjoint functors between them. Using that colimits in $\mathrm{Pr}^L \simeq (\mathrm{Pr}^R)^{op}$ are computed by passing to right adjoints and then computing the limit in $\widehat{\mathrm{Cat}}_\infty$, see \cite[Theorem 5.5.3.18.]{luriehtt}, we get equivalently that the forget functor
$$\mathrm{LTop} \rightarrow \mathrm{Pr}^L$$
preserves filtered colimits, where $\mathrm{LTop} \simeq \mathrm{RTop}^{op}$ is the (large) $\infty$-category of $\infty$-topoi and left adjoints $f*$ of geometric morphisms between them.

\begin{theorem} \label{sheavesfiltered}
The functor
$$ \mathrm{Sh}((-,fin),\mathrm{Spc}) : \mathrm{DLatt}_{bd} \rightarrow \mathrm{Pr}^{L}_{\omega} $$
preserves filtered colimits.
\end{theorem}

\begin{proof}
The composite of the functors
$$ \mathrm{DLatt}_{bd} \rightarrow \mathrm{Frm} \rightarrow \mathrm{LTop} \rightarrow \mathrm{Pr}^L$$
preserve filtered colimits. This is because the first two functors are left adjoints by Theorem \ref{stonedualitycoherent} and Theorem \ref{adjunctionlocalestopoi}, and the third preserves filtered colimits by Theorem \ref{cofilteredlimitstopoi}.
By Corollary \ref{compgen} and Lemma \ref{compmor} it lifts to $\mathrm{Pr}^{L}_{\omega}$ along the colimit preserving functor $\mathrm{Pr}^{L}_{\omega} \rightarrow \mathrm{Pr}^L$ (by Lemma \ref{collectionofcolimitspresentable}), hence the claim follows.
\end{proof}

The $\infty$-category $\mathrm{Pr}^L$ comes with a symmetric monoidal structure referred to as the \emph{Lurie tensor product} $\otimes$, which preserves colimits in both variables, see \cite[4.8.1]{lurieha}. Given a site $(C, \tau)$ and a presentable $\infty$-category $\mathcal{E}$, we define the $\infty$-category of $\mathcal{E}$-valued sheaves as
$$\mathrm{Sh}((C, \tau), \mathcal{E}) = \mathrm{Sh}(C, \tau) \otimes \mathcal{E}.$$

\begin{example}
A pleasing special case is obtained by taking $\mathcal{E} = \mathbf{2}$, which is a presentable poset. In this case $\mathrm{Sh}(C, \tau) \otimes \mathbf{2}$ recovers the frame of propositional sheaves on $C$, hence our choice of notation does not clash. More generally, the functor $\mathrm{Sub}(1) : \mathrm{RTop} \rightarrow \mathrm{Loc}$ can be identified with the functor $- \otimes \mathbf{2}$.
\end{example}

\begin{example} \label{sheavesfiniteframepresentable}
Suppose $D$ is a finite frame and $\mathcal{C}$ a presentable $\infty$-category. Then it follows directly from Proposition \ref{sheavesonfiniteframe} that there is a natural equivalence
$$\mathrm{Sh}(D, \mathcal{C}) \simeq \mathrm{Fun}( \mathrm{pts}(D)^{op}, \mathcal{C} )$$
obtained by restricting a sheaf $G : D \rightarrow \mathcal{C}$ along the inclusion $\mathrm{pts}(D) \rightarrow D$.
\end{example}

\begin{remark}
As a further remark on the previous example, if $P$ is a finite poset and $\mathcal{C} = \mathrm{RMod}_R$ is the $\infty$-category of right $R$-modules for an $E_1$-ring spectrum $R$, we have that
$$\mathrm{Fun}( P^{op}, \mathrm{RMod}_R )$$
is compactly generated by the finite set of representables $R \otimes y_p$ for $p \in P$. Let $I(P, R)$ be the endomorphism ring spectrum of the compact (!) object $\bigoplus_{p \in P} y_p \otimes R $. It follows by the Schwede-Shipley recognition theorem, see \cite[Theorem 7.1.2.1]{lurieha}, that
$$\mathrm{Fun}( P^{op}, \mathrm{RMod}_R ) \simeq \mathrm{RMod}_{I(P, R)}.$$
An inspection reveals that the underlying $R$-module of $I(P, R)$ is simply the free $R$-module generated by the set $P_1 = \{(p,q) \in P^2 ~|~ p \leq q \}$. In the case of a discrete ring $R$, this recovers the classical incidence algebra of a finite poset, as originally defined by Rota, \cite{Rota1964}. See also \cite{SpiegelOdonnell1997}.
\end{remark}

\subsection{Dualizable $\infty$-categories}

We now turn to the stable context. Let $\mathrm{Pr}^L_{dual}$ be the full subcategory of $\mathrm{Pr}^L_{ca}$ spanned by \emph{stable} compactly assembled presentable $\infty$-categories, also called \emph{dualizable} $\infty$-categories. We cite two structural results about this $\infty$-category.

\begin{proposition}[\cite{krause_nikolaus_puetzstueck}, Prop 2.19.9.]
There is an adjunction
\[\begin{tikzcd}
	{\mathrm{Pr}^L_{ca}} & {\mathrm{Pr}^L_{dual}}
	\arrow[""{name=0, anchor=center, inner sep=0}, "{- \otimes \mathrm{Sp}}", curve={height=-12pt}, from=1-1, to=1-2]
	\arrow[""{name=1, anchor=center, inner sep=0}, curve={height=-12pt}, hook', from=1-2, to=1-1]
	\arrow["\dashv"{anchor=center, rotate=-90}, draw=none, from=0, to=1]
\end{tikzcd}\]
where the right adjoint is given as the (fully faithful) forget functor, and the left adjoint is given as the Lurie tensor product with the $\infty$-category of spectra. Furthermore, both $\infty$-categories equipped with the Lurie tensor product are symmetric monoidal $\infty$-categories such that $-\otimes-$ preserves colimits in each variable, and both the left and right adjoint functors in the diagram above are strong symmetric monoidal functors.
\end{proposition}

\begin{proposition}[See \cite{efimov2025ktheorylocalizinginvariantslarge},  Proposition 1.65.]
Colimits in the $\infty$-category $\mathrm{Pr}^L_{dual}$ exist and the forget functor $\mathrm{Pr}^L_{dual} \rightarrow \mathrm{Pr}^L$ preserves them.
\end{proposition}

The following statement about filtered colimits of \emph{bounded} distributive lattices is an immediate consequence of Theorem \ref{sheavesfiltered}, given the above two propositions.

\begin{corollary} \label{sheavesfiltereddualizable}
Let $\mathcal{C}$ be a dualizable, stable $\infty$-category. The functor
$$ \mathrm{Sh}((-,fin),\mathcal{C}) : \mathrm{DLatt}_{bd} \rightarrow \mathrm{Pr}^L_{dual} $$
preserves filtered colimits.
\end{corollary}

We would like to generalize Corollary \ref{sheavesfiltereddualizable} to lower bounded distributive lattices. The problem here is that homomorphisms of lower bounded distributive lattices do not necessarily induce geometric morphisms between their sheaf topoi (To be more precise, they produce \emph{partially defined} morphisms). In order to overcome this obstacle, we can use a trick. First, let us define what an exact sequence of dualizable $\infty$-categories is.

\begin{definition}
An exact sequence in $\mathrm{Pr}^L_{dual}$ is a fiber-cofiber sequence
$$ \mathcal{C} \xrightarrow{f_!} \mathcal{D} \xrightarrow{g^*} \mathcal{E}.$$
Concretely, this means that $f_!$ is fully faithful and $g^*$ is the (unique) Bousfield localization with $\mathrm{ker}(g^*) = \mathcal{C}$.
\end{definition}

\begin{example}
Let $L$ be locale and $U$ an open subset. Let $\mathcal{C}$ be a stable presentable $\infty$-category. Then the open-closed decomposition discussed in subsection \ref{generalitiestopoi} gives the sequence
$$ \mathrm{Sh}(L_{/U}, \mathcal{C}) \xrightarrow{i_!} \mathrm{Sh}(L, \mathcal{C}) \xrightarrow{c^*} \mathrm{Sh}(L_{U/}, \mathcal{C}).$$
If $\mathrm{Sh}(L, \mathcal{C})$ and $\mathrm{Sh}(L_{/U}, \mathcal{C})$ are in $\mathrm{Pr}^L_{dual}$ then so is $\mathrm{Sh}(L_{U/}, \mathcal{C})$ and the above sequence is an exact sequence in $\mathrm{Pr}^L_{dual}$.
\end{example}

We will need the following lemma due to Efimov.

\begin{lemma}[\cite{efimov2025ktheorylocalizinginvariantslarge}, Proposition 1.67] \label{efimovslemma}
A filtered colimit of short exact sequences in $\mathrm{Pr}^L_{dual}$ is again a short exact sequence.
\end{lemma}

Using this result, we can extrapolate to get a version of Corollary \ref{sheavesfiltereddualizable} for the lower bounded case as well.

\begin{theorem} \label{sheavesfiltereddualizablelowerbounded}
Let $\mathcal{C}$ be a dualizable, stable $\infty$-category. The functor
$$ \mathrm{Sh}((-,fin),\mathcal{C}) : \mathrm{DLatt}_{lb} \rightarrow \mathrm{Pr}^L_{dual} $$
preserves filtered colimits. In particular, for any finitary localizing invariant $F : \mathrm{Pr}^L_{dual} \rightarrow \mathcal{E}$, the functor
$$F( \mathrm{Sh}((-,fin),\mathcal{C})) : \mathrm{DLatt}_{lb} \rightarrow \mathcal{E}$$ preserves filtered colimits.
\end{theorem}

\begin{proof}
It suffices to prove the theorem for the case $\mathcal{C} = \mathrm{Sp}$. The filtered colimit preserving functor
$$ \mathrm{Sh}((-,fin),\mathrm{Sp}) : \mathrm{DLatt}_{bd} \rightarrow  \mathrm{Pr}^L_{dual} $$
canonically produces a filtered colimit presering functor
$$ \mathrm{Sh}((-,fin),\mathrm{Sp}) : {\mathrm{DLatt}_{bd}}_{/\mathbf{2}} \rightarrow  {\mathrm{Pr}^L_{dual}}_{/ \mathrm{Sp}}. $$
Precomposing with the fully faithful left adjoint $(-)_\infty : \mathrm{DLatt}_{lb} \rightarrow {\mathrm{DLatt}_{bd}}_{/\mathbf{2}}$ still preserves filtered colimits. Now, for a given lower bounded distributive lattice $D$ we have the short exact sequence
$$ \mathrm{Sh}((D, fin), \mathrm{Sp} ) \rightarrow \mathrm{Sh}((D_\infty, fin), \mathrm{Sp} ) \rightarrow \mathrm{Sp},$$
given by the open-closed decomposition of $\mathrm{Ind}(D_\infty)$, see Display \ref{openclosedtopelement}, hence we can identify the functor $\mathrm{Sh}((-, fin), \mathrm{Sp} )$ we are looking for with the composite
$$\mathrm{DLatt}_{lb} \xrightarrow{(-)_\infty} {\mathrm{DLatt}_{bd}}_{/\mathbf{2}} \rightarrow  {\mathrm{Pr}^L_{dual}}_{/ \mathrm{Sp}} \xrightarrow{\mathrm{ker}}  \mathrm{Pr}^L_{dual}. $$
By Lemma \ref{efimovslemma}, short exact sequences are stable under filtered colimits, and hence we conclude the theorem.
\end{proof}

\subsection{Localizing invariants and semiorthogonal decompositions}

We will follow the definition of localizing invariants laid out in \cite{blumberg_gepner_tabuada}. Let $\mathrm{Cat}^{perf}$ be the $\infty$-category of small, stable, idempotent complete $\infty$-categories. Note that $\mathrm{Cat}^{perf}$ is pointed, as the zero category $0$ is both the initial and terminal stable $\infty$-category. An \emph{exact sequence}, also Verdier sequence, is defined to be a fiber-cofiber sequence
$$\mathcal{A} \rightarrow \mathcal{B} \rightarrow \mathcal{C}.$$
Let $\mathcal{E}$ be a stable $\infty$-category. A \emph{localizing invariant} is a functor $F : \mathrm{Cat}^{perf} \rightarrow \mathcal{E}$ that preserves the zero object  and fiber-cofiber sequences. If $\mathcal{A}$ is furthermore presentable, we call $F$ a \emph{finitary} localizing invariant if $F$ preserves filtered colimits.

\begin{example}
Important examples of localizing invariants as discussed in \cite{blumberg_gepner_tabuada} are non-connective, algebraic $K$-theory and topological Hochschild homology,
$$\begin{array}{rcl}
K : \mathrm{Cat}^{perf} & \rightarrow & \mathrm{Sp} \\ 
\mathrm{THH} : \mathrm{Cat}^{perf} & \rightarrow & \mathrm{Sp}
\end{array}$$
as well as the universal localizing invariant
$$ \mathcal{U}_{loc} : \mathrm{Cat}^{perf} \rightarrow \mathrm{Mot}_{loc}.$$
All three of these are finitary invariants. An example of a non-finitary localizing invariant is topological cyclic homology $\mathrm{TC}$.
\end{example}

By analogy we call $F : \mathrm{Pr}^{L}_{dual} \rightarrow \mathcal{E}$ a \emph{localizing invariant of dualizable $\infty$-categories} if $F$ preserves the zero object and sends exact sequences to fiber sequences. Note that we have a functor $\mathrm{Ind} : \mathrm{Cat}^{perf} \rightarrow \mathrm{Pr}^{L}_{dual}$. The following theorem is due to Efimov.

\begin{theorem}[\cite{efimov2025ktheorylocalizinginvariantslarge}, Theorem 4.10]
Let $\mathcal{E}$ be a presentable, stable $\infty$-category. Then precomposition with $\mathrm{Ind}$ induces an equivalence
$$ \mathrm{Fun}_{loc, \omega}( \mathrm{Pr}^{L}_{dual}, \mathcal{E} ) \xrightarrow{ \cong } \mathrm{Fun}_{loc, \omega}( \mathrm{Cat}^{perf}, \mathcal{E} )$$
between the $\infty$-categories of finitary localizing invariants on $\mathrm{Pr}^{L}_{dual}$ and $\mathrm{Cat}^{perf}$ respectively.
\end{theorem}

We remark that the same results holds true when relaxing the requirement to $\kappa$-accessible localizing invariants for any regular cardinal $\kappa$. From here on, we will implicitly always identify any finitary localizing invariant with its extension to dualizable $\infty$-categories. We collect some results about the interaction of finitary localizing invariants with semi-orthogonal decompositions.

\begin{definition}[\cite{efimov2025ktheorylocalizinginvariantslarge}, Definition 1.80.]
Let $\mathcal{C}$ be a presentable stable $\infty$-category and $I$ a poset. An $I$-indexed semi-orthogonal decomposition of $\mathcal{C}$ is a collection of subcategories $\mathcal{C}_i, i \in I$ such that:
\begin{enumerate}
\item The inclusions $\mathcal{C}_i \subset \mathcal{C}$ are strongly continuous.
\item For $i \not\leq j, x \in \mathcal{C}_i, y \in \mathcal{C}_j$ we have
$$ \mathrm{Hom}_{\mathcal{C}}(x,y) = 0.$$
\item The categories $\mathcal{C}_i$ generate $\mathcal{C}$ as a localizing subcategory.
\end{enumerate} 
\end{definition}
We remark that in this situation, $\mathcal{C}$ is dualizable iff $\mathcal{C}_i$ are dualizable for all $i \in I$. The reason this notion is useful is the following.

\begin{proposition}[\cite{efimov2025ktheorylocalizinginvariantslarge}, Proposition 4.14.] \label{ktheorysemiorthogonal}
Let $\mathcal{C}$ be a dualizable $\infty$-category admitting a semi-orthogonal decomposition $C = \langle \mathcal{C}_i, i \in I \rangle$ for some poset $I$, and let $F : \mathrm{Cat}^{dual} \rightarrow \mathcal{E}$ be a finitary localizing invariant. Then the natural map
$$ \bigoplus_{i \in I} F( \mathcal{C}_i ) \rightarrow F(\mathcal{C})$$
is an isomorphism.
\end{proposition}

\begin{corollary} \label{semiorthogonaldecomposition}
Let $P$ be a poset, $\mathcal{C}$ a dualizable $\infty$-category and $F : \mathrm{Cat}^{perf} \rightarrow \mathcal{E}$ a finitary localizing invariant. Then
$$F^{cont}( \mathrm{PSh}(P; \mathcal{C}) ) \cong \bigoplus_{p \in P} F^{cont}( \mathcal{C} ).$$
\end{corollary}

\begin{proof} We claim that we have a $P$-indexed semi-orthogonal decomposition of $\mathrm{PSh}(P; \mathcal{C})$. It suffices to do the case for $\mathcal{C} = \mathrm{Sp}$. For $p \in P$ let
\[\begin{tikzcd}
	{\mathrm{Sp}} & {\mathrm{PSh}(P;\mathrm{Sp})}
	\arrow["{p_!}", curve={height=-12pt}, hook, from=1-1, to=1-2]
	\arrow["{p^*}", curve={height=-12pt}, from=1-2, to=1-1]
\end{tikzcd}\]
be the adjunction induced by the fully faithful functor $ p : \mathrm{pt} \rightarrow P$ that selects $p$. We thus get the $P$-indexed family of subcategories $p_!(\mathrm{Sp}) \subset \mathrm{PSh}(P; \mathrm{Sp} )$. These inclusions are all strongly continuous, as $p^*$ has the further right adjoint $p_*$. The image of $p_!(\mathbb{S})$ is the representable functor at $p$, hence we have that all of the $p_!(\mathrm{Sp})$ generate $\mathrm{PSh}(P; \mathrm{Sp} )$. Lastly, we have that $p^* q_! (Y) = q_!(Y)(p) = 0$ for $p \not\leq q$ by the standard formula for the value of left Kan extension, hence
$$ \mathrm{Hom}_{\mathrm{PSh}(P; \mathrm{Sp} )}(p_!(X), q_!(Y)) = \mathrm{Hom}_{\mathrm{Sp}}(X, p^* q_!(Y)) = 0.$$
\end{proof}

\section{The ring spectrum of motives} \label{functoriality}

We will also need a higher algebraic version of the module of motives. Given an abelian group $A$, write $HA$ for the Eilenberg-Maclane spectrum of $A$. We call a connective spectrum $X$ a \emph{Moore spectrum} for $A$ if $X \otimes H \mathbb{Z} \simeq HA$. Moore spectra for any given abelian group $A$ exist and are unique up to equivalence. However, there exists no functor from abelian groups to spectra realizing the Moore spectrum of an abelian group. The Moore spectrum $X$ of a free abelian group $A$ is given as a wedge sum of sphere spectra $\bigoplus \mathbb{S}$, indexed by a choice of basis of $A$. Observe that the subcategory $ \mathrm{MooreSp} \subset \mathrm{Sp}$ spanned by Moore spectra (for arbitrary abelian groups) is closed under filtered colimits, as tensoring with $H \mathbb{Z}$ and taking homotopy groups preserves filtered colimits. For more details, see \cite[II. 6.3]{schwede_symmetric}.

\begin{definition}
Let $D$ be a lower bounded distributive lattice. The \emph{spectrum of $D$-motives} $\mathcal{M}(D)$ is defined to be the  Moore spectrum associated to the module of motives $M(D)$.
\end{definition}
This definition has two issues:
\begin{itemize}
\item A priori, it is not clear if this is a functorial assignment, as the formation of Moore spectra is not functorial.
\item The module of motives $M(D)$ has a (well-defined and non-unital) commutative ring structure given by $[U] \cdot [V] = [U \wedge V]$. If $D$ has a top element as well, then $M(D)$ is a commutative ring with unit given by $[1]$. This structure carries over to the Moore spectrum $\mathcal{M}(D)$, a fact that is not obvious from the definition we gave.
\end{itemize}
We will rectify both of these shortcomings in this section. Before we do so, a few comments on algebra in a higher categorical setting. We call an $\infty$-category $\mathcal{C}$ with finite products \emph{cartesian}.

\begin{definition}[See also \cite{joyal2008}, \cite{cranch2010algebraic}, \cite{berman2019higher}, and \cite{Gepner_2015}, Appendix B]
An \emph{algebraic theory}, also called \emph{Lawvere theory}, is a cartesian $\infty$-category $\mathcal{L}$ together with a given object $x \in \mathcal{L}$ that generates $\mathcal{L}$ under finite products. If $\mathcal{C}$ is a cartesian category, a \emph{model} $M$ for $\mathcal{L}$ is a finite-product preserving functor $\mathcal{L} \rightarrow \mathcal{C}$. Write $\text{Mod}_\mathcal{L}(\mathcal{C})$ for the full subcategory of $\Fun(\mathcal{L}, \mathcal{C})$ consisting of the models.
\end{definition}

\begin{example}
We have the following useful Lawvere theories.
\begin{itemize}
\item The Lawvere theory for Boolean algebras is given by $\mathrm{Fin}_2$, the category of finite sets of cardinality $2^n$ for some $n \in \mathbb{N}$.
\item The Lawvere theory for bounded distributive lattices is given by the subcategory $\mathrm{Cubes}$ of $\mathrm{FinPoset}$ spanned by \emph{combinatorial cubes}, which are posets of the form $\mathbf{2}^n$ for some $n \in \mathbb{N}$, with $\mathbf{2} = \{0 \leq 1\}$.
\item The Lawvere theory for lower bounded distributive lattices is given by the subcategory $\mathrm{Cubes}_{root}$ of $\mathrm{FinPoset}_{*/}$ of finite pointed posets, spanned by \emph{rooted combinatorial cubes}, which are the combinatorial cubes $\mathbf{2}^n$ for some $n \in \mathbb{N}$, together with the chosen basepoint being the initial object.
\item The Lawvere theory for join-semilattices is given by the category $\mathrm{FinRel}$ of finite sets with relations as morphisms. This can be seen by observing that the free join-semilattice on $n$ generators is given by the powerset $\mathcal{P}(\mathbf{n})$ of a $n$-element set, and therefore a homomorphism  $\mathcal{P}(\mathbf{n}) \rightarrow \mathcal{P}(\mathbf{m})$ is determined by a choice of subset of $\mathbf{n}$ for each $i \in \mathbf{n}$, in other words a relation.
\end{itemize}
The Set-valued models for the above four Lawvere theories agree with the categories of Boolean algebras, bounded distributive lattices and lower bounded distributive lattices. Observe that we have cartesian functors
$$\mathrm{FinRel} \xrightarrow{\mathcal{P}} \mathrm{Cubes}_{root} \xrightarrow{forget} \mathrm{Cubes} \xrightarrow{forget} \mathrm{Fin}_2,$$
which correspond to the fact that the theories of join-semilattices, lower bounded distributive lattices, bounded distributive lattices and Boolean algebras are built successively from each other by adding operations and equations.
\end{example}

Consider the $\infty$-category $\text{Mod}_\mathcal{L}(\mathrm{Spc})$. The Yoneda embedding provides a functor
$$y : \mathcal{L}^{op} \rightarrow \text{Mod}_\mathcal{L}(\mathrm{Spc})$$
which is finite coproduct preserving, which we think of as the embedding of \emph{finitely generated free} models of $\mathcal{L}$. This functor has a useful universal property.

\begin{theorem}
Let $\mathcal{L}$ be a Lawvere theory, and $\mathcal{E}$ a category with small colimits. Then precomposition with $y$ induces an equivalence of categories
$$\mathrm{Fun}^L(  \mathrm{Mod}_\mathcal{L}(\mathrm{Spc}) , \mathcal{E} ) \simeq  \mathrm{Fun}^\amalg( \mathcal{L}^{op}, \mathcal{E} )$$
where $\mathrm{Fun}^L$ denotes the $\infty$-category of colimit preserving functors, and $\mathrm{Fun}^\amalg$ denotes the $\infty$-category of finite coproduct preserving functors.
\end{theorem}

\begin{proof}
This is more or less directly handled by Lurie in \cite[5.3.6.]{luriehtt}. Lurie shows that for a small category $C$ with finite coproducts there exists a functor $C \rightarrow \mathcal{P}_\Sigma(C)$ with the above universal property. To see that $\mathcal{P}_\Sigma(C)$ must agree with $\text{Mod}_\mathcal{L}(\mathrm{Spc})$ for $C = \mathcal{L}^{op}$, just observe that
$$\begin{array}{rlll} \mathcal{P}_\Sigma(C)^{op} & \simeq \mathrm{Fun}^L( \mathrm{Spc}, \mathcal{P}_\Sigma(C)^{op}) & \simeq \mathrm{Fun}^R( \mathcal{P}_\Sigma(C)^{op},  \mathrm{Spc} )^{op} & \\ & \simeq \mathrm{Fun}^L( \mathcal{P}_\Sigma(C), \mathrm{Spc}^{op} ) & \simeq  \mathrm{Fun}^\amalg( \mathcal{L}^{op}, \mathrm{Spc}^{op} ) & \hspace{-5ex} \simeq \text{Mod}_\mathcal{L}(\mathrm{Spc})^{op}.
\end{array}$$
\end{proof} 
In other words, any left adjoint functor from $\text{Mod}_\mathcal{L}(\mathrm{Spc})$ to any other cocomplete $\infty$-category $\mathcal{E}$ is induced by a co-model of $\mathcal{L}$ in $\mathcal{E}$. To be more precise, if $f : \mathcal{L}^{op} \rightarrow \mathcal{E}$ is finite coproduct preserving, then we have the induced adjunction
\[\begin{tikzcd}
	{\text{Mod}_\mathcal{L}(\mathrm{Spc})} & {\mathcal{E}}
	\arrow[""{name=0, anchor=center, inner sep=0}, "{f_!}", curve={height=-12pt}, from=1-1, to=1-2]
	\arrow[""{name=1, anchor=center, inner sep=0}, "{f^*}", curve={height=-12pt}, from=1-2, to=1-1]
	\arrow["\dashv"{anchor=center, rotate=-90}, draw=none, from=0, to=1]
\end{tikzcd}\]
where the right adjoint functor $\mathcal{E} \rightarrow \text{Mod}_\mathcal{L}(\mathrm{Spc})$ is given by
$$\begin{array}{rcl}
f^* : \mathcal{E} & \rightarrow & \text{Mod}_\mathcal{L}(\mathrm{Spc}) \\
e & \mapsto & \mathrm{Map}_\mathcal{E}( f(-), e )
\end{array}$$

\begin{example}
Let $f : \mathcal{L} \rightarrow \mathcal{L}'$ be a cartesian functor between Lawvere theories. By considering the composite
$$\mathcal{L}^{op} \xrightarrow{f} (\mathcal{L}')^{op} \xrightarrow{y} \text{Mod}_\mathcal{L'}(\mathrm{Spc}) $$
we get an induced adjunction
\[\begin{tikzcd}
	{\text{Mod}_\mathcal{L}(\mathrm{Spc})} & {\text{Mod}_\mathcal{L'}(\mathrm{Spc}).}
	\arrow[""{name=0, anchor=center, inner sep=0}, "{f_!}", curve={height=-12pt}, from=1-1, to=1-2]
	\arrow[""{name=1, anchor=center, inner sep=0}, "{f^*}", curve={height=-12pt}, from=1-2, to=1-1]
	\arrow["\dashv"{anchor=center, rotate=-90}, draw=none, from=0, to=1]
\end{tikzcd}\]
\end{example}

The Lawvere theory of Boolean algebras has a surprising feature: Any model for it in the $\infty$-category of spaces is automatically discrete.\footnote{We thank Maxime Ramzi and David Wärn for a discussion leading to this argument.} 

\begin{theorem}
Let $B : \mathrm{Fin}_2 \rightarrow \mathrm{Spc}$ be a model for the theory of Boolean algebras. Then the underlying space $B(\mathbf{2})$ is discrete. Moreover, there is an equivalence of $\infty$-categories
$$ \mathrm{BoolAlg} = \mathrm{Mod}_{\mathrm{BoolAlg}}(\mathrm{Set}) \simeq \mathrm{Mod}_{\mathrm{BoolAlg}}(\mathrm{Spc}).$$
\end{theorem}

\begin{remark}
The same result was obtained independently by Benjamin Antieau in \cite{antieau2025animationoppositefinitesets}.
\end{remark}

\begin{proof} Denote the one and two-element set by $\mathbf{1}$ and $\mathbf{2}$, respectively. The statement about the equivalence of $\infty$-categories follows immediately by observing that if $B$ is a Boolean algebra object such that $B(\mathbf{2})$ is discrete, then the value $B(\mathbf{2}^n) = B(\mathbf{2})^n$ is discrete for any $n \geq 0$.

Now let $B$ be a Boolean algebra object in spaces. First note that restricting to the $(0, \vee)$-structure equips $B$ with the structure of an $E_\infty$-monoid. By the Eckmann-Hilton argument, the operation $\vee$ agrees with the usual group multiplication on $\pi_n( B(\mathbf{2}), 0 )$ for any $n \geq 1$. However, for the operation $\vee$ any element is idempotent, hence $\pi_n( B(\mathbf{2}), 0 ) = 1$ is the trivial group. In other words, the map $0 : \mathrm{pt} \simeq B(\mathbf{1}) \rightarrow B(\mathbf{2})$ is a monomorphism in the $\infty$-category of spaces.

Observe that the diagonal map $\Delta : \mathbf{2} \rightarrow \mathbf{2} \times \mathbf{2}$ is equivalent to 
the map $(\mathrm{id},0) : \mathbf{2} \rightarrow \mathbf{2} \times \mathbf{2}$ in the category $\mathrm{Fin}_2$, as both are injections. Therefore, the diagonal map $\Delta : B(\mathbf{2}) \rightarrow B(\mathbf{2}) \times B(\mathbf{2})$ is obtained via pullback from the monomorphism $0 : B(\mathbf{1}) \rightarrow B(\mathbf{2})$, and hence itself a monomorphism. But a space is discrete iff its diagonal map is a monomorphism, see \cite[Lemma 5.5.6.17.]{luriehtt}.
\end{proof}

Now let $(\mathcal{C},\otimes)$ be a presentably symmetric monoidal $\infty$-category, i.e. $(\mathcal{C},\otimes) \in \mathrm{CAlg}( \mathrm{Pr}^L )$, and consider the commutative algebra  $1_\mathcal{C} \in \mathrm{CAlg}(\mathcal{C})$ given by the unit. Assume furthermore that the tensor product preserves finite products in each variable. By cotensoring with a finite set we get a functor
$$\begin{array}{rcl}
1_\mathcal{C}^- : \mathrm{Fin}_2 & \rightarrow & \mathrm{CAlg}(\mathcal{C}) \\
N & \mapsto & 1_\mathcal{C}^N. 
\end{array}$$
(Note that the forget functor $\mathrm{CAlg}(\mathcal{C}) \rightarrow \mathcal{C}$ preserves limits \cite[Corollary 3.2.2.5.]{lurieha}.)
The functor $1_\mathcal{C}^-$ preserves coproducts since by \cite[3.2.4]{lurieha} coproducts in $\mathrm{CAlg}(\mathcal{C})$ are given by tensor products, and under the assumption that the tensor product commutes with finite products we have that the natural map
$$ 1_\mathcal{C}^N \otimes 1_\mathcal{C}^M \rightarrow 1_\mathcal{C}^{N \times M}$$
is an equivalence. As a corollary we get an induced adjunction
\[\begin{tikzcd}
	 {\mathrm{BoolAlg}} \simeq {\text{Mod}_{\mathrm{BoolAlg}}(\mathrm{Spc})} \hspace{5ex} & {\mathrm{CAlg}(\mathcal{C})}
	\arrow[""{name=0, anchor=center, inner sep=0}, "{L}", curve={height=-12pt}, from=1-1, to=1-2]
	\arrow[""{name=1, anchor=center, inner sep=0}, "{\mathrm{Idem}}", curve={height=-12pt}, from=1-2, to=1-1]
	\arrow["\dashv"{anchor=center, rotate=-90}, draw=none, from=0, to=1]
\end{tikzcd}\]
where the right adjoint $\mathrm{Idem}$ associates to a commutative algebra object $A$ its (necessarily discrete) space of idempotent elements. (We remark that in the classical case of $\mathcal{C} = \mathrm{Ab}$, this recovers the Boolean algebra of idempotents of a commutative ring.)
Now we simply apply this to the situation where $\mathcal{C} = \mathrm{Sp}$ and the unit is the sphere spectrum $\mathbb{S}$.

\begin{proposition}
There is an adjunction
\[\begin{tikzcd}
	{\mathrm{BoolAlg}} & {\mathrm{CAlg}(\mathrm{Sp})}
	\arrow[""{name=0, anchor=center, inner sep=0}, "{\mathcal{M}}", curve={height=-12pt}, from=1-1, to=1-2]
	\arrow[""{name=1, anchor=center, inner sep=0}, "{\mathrm{Idem}}", curve={height=-12pt}, from=1-2, to=1-1]
	\arrow["\dashv"{anchor=center, rotate=-90}, draw=none, from=0, to=1]
\end{tikzcd}\]
such that for a Boolean algebra $B$, the value $\mathcal{M}(B)$ is the Moore spectrum to the module of motives $M(B)$.
\end{proposition}

\begin{remark}
The right adjoint sends a given $E_\infty$-ring spectrum $A$ to the discrete space $\mathrm{Idem}(A) = \mathrm{Map}_{\mathrm{CAlg}}( \mathbb{S} \times \mathbb{S}, A)$, viewed as a Boolean algebra. A point of this space, given by a ring map $f : \mathbb{S} \times \mathbb{S} \rightarrow A$, can be interpreted as a pair of complementary, idempotent, elements $p,q$ in $\pi_0(A)$, in the sense that the relations
$$p^2 = q^2 = 1, ~ pq = 0, \text{ and } p+q =1$$
hold.
\end{remark}

\begin{proof}
We need to verify that the value of the induced left adjoint $L : \mathrm{BoolAlg} \rightarrow \mathrm{CAlg}(\mathrm{Sp})$ on a Boolean algebra $B$ agrees with $\mathcal{M}(B)$, defined as the Moore spectrum. We do this in two steps.\footnote{We thank Benjamin Dünzinger for suggesting this argument.}
\begin{itemize}
\item Since every finite set is a retract of a set of order a power of $2$, we observe that $L$ agrees with the cotensor
$$\begin{array}{rcl}
\mathrm{FinSet}^{op} & \rightarrow & \mathrm{Sp} \\
F & \mapsto & \mathbb{S}^F.
\end{array}
$$
when restricted to finite Boolean algebras, under the equivalence $\mathrm{FinBoolAlg} \simeq \mathrm{FinSet}^{op}$. 
\item Arbitrary Boolean algebras are obtained as filtered colimits of finite Boolean algebras inside the category $\mathrm{BoolAlg}$. Therefore observe that $L$ agrees with the induced filtered colimit preserving functor
$$ \mathrm{BoolAlg} \cong \mathrm{Ind}(\mathrm{FinSet}^{op}) \cong \mathrm{Pro}(\mathrm{FinSet})^{op} \rightarrow \mathrm{Sp}.$$
This matches the value-wise definition by Lemma \ref{motivesfiltered}.
\end{itemize}
\end{proof}

We can now compose with the adjunction
\[\begin{tikzcd}
	{\mathrm{DLatt}_{bd}} & {\mathrm{BoolAlg}}.
	\arrow[""{name=0, anchor=center, inner sep=0}, "{\mathrm{Bool}}", curve={height=-12pt}, from=1-1, to=1-2]
	\arrow[""{name=1, anchor=center, inner sep=0}, "{forget}", curve={height=-12pt}, from=1-2, to=1-1]
	\arrow["\dashv"{anchor=center, rotate=-90}, draw=none, from=0, to=1]
\end{tikzcd}\]
to get $\mathcal{M}$ as a functor on bounded distributive lattices. This matches the value-wise definition by Proposition \ref{motivesonbooleanization}.

\begin{proposition}
There is an adjunction
\[\begin{tikzcd}
	{\mathrm{DLatt}_{bd}} & {\mathrm{CAlg}(\mathrm{Sp})}
	\arrow[""{name=0, anchor=center, inner sep=0}, "{\mathcal{M}}", curve={height=-12pt}, from=1-1, to=1-2]
	\arrow[""{name=1, anchor=center, inner sep=0}, "{\mathrm{Idem}}", curve={height=-12pt}, from=1-2, to=1-1]
	\arrow["\dashv"{anchor=center, rotate=-90}, draw=none, from=0, to=1]
\end{tikzcd}\]
such that for a bounded distributive lattice $D$, the value $\mathcal{M}(D)$ is the Moore spectrum to the module of motives $M(D)$.
\end{proposition}
	
In order to define $\mathcal{M}$ as a functor on lower bounded distributive lattice, we need the concept of a non-unital algebra.

\begin{proposition}[\cite{lurieha}, Proposition 5.4.4.10.] \label{nonunital}
Let $(\mathcal{C},\otimes)$ be a stable, symmetric monoidal  $\infty$-category. There is an equivalence
$$ \mathrm{CAlg}^{nu}(\mathcal{C}) \simeq \mathrm{CAlg}(\mathcal{C})_{/\mathbb{S}},$$
where the left-hand side is the $\infty$-category of non-unital commutative algebras in $\mathcal{C}$. The inverse functor sends an augmented algebra $f : A \rightarrow \mathbb{S}$ to the ideal $\mathrm{fib}(f)$.
\end{proposition}

\begin{theorem} \label{spaceofidempotents}
There is an adjunction
\[\begin{tikzcd}
	{\mathrm{DLatt}_{lb}} & {\mathrm{CAlg}^{nu}(\mathrm{Sp})}
	\arrow[""{name=0, anchor=center, inner sep=0}, "{\mathcal{M}}", curve={height=-12pt}, from=1-1, to=1-2]
	\arrow[""{name=1, anchor=center, inner sep=0}, "{\mathrm{Idem}}", curve={height=-12pt}, from=1-2, to=1-1]
	\arrow["\dashv"{anchor=center, rotate=-90}, draw=none, from=0, to=1]
\end{tikzcd}\]
such that for a lower bounded distributive lattice $D$, the value $\mathcal{M}(D)$ is the Moore spectrum to the module of motives $M(D)$.
\end{theorem}

\begin{proof}
For formal reasons, the adjunction for bounded distributive lattices induces an adjunction
\[\begin{tikzcd}
	{{\mathrm{DLatt}_{bd}}_{/ \mathbf{2}}} & {\mathrm{CAlg}(\mathrm{Sp})_{/\mathbb{S}}}
	\arrow[""{name=0, anchor=center, inner sep=0}, "{\mathcal{M}}", curve={height=-12pt}, from=1-1, to=1-2]
	\arrow[""{name=1, anchor=center, inner sep=0}, "{\mathrm{Idem}}", curve={height=-12pt}, from=1-2, to=1-1]
	\arrow["\dashv"{anchor=center, rotate=-90}, draw=none, from=0, to=1]
\end{tikzcd}\]
since $\mathcal{M}(\mathbf{2}) \simeq \mathbb{S}$ and $\mathrm{Idem}(\mathbb{S}) \simeq \mathbf{2}$. Using the adjunction
\[\begin{tikzcd}
	{\mathrm{DLatt}_{lb}} & {{\mathrm{DLatt}_{bd}}_{/\mathbf{2}}}
	\arrow[""{name=0, anchor=center, inner sep=0}, "{(-)_\infty}", curve={height=-12pt}, hook, from=1-1, to=1-2]
	\arrow[""{name=1, anchor=center, inner sep=0}, "R", curve={height=-12pt}, from=1-2, to=1-1]
	\arrow["\dashv"{anchor=center, rotate=-90}, draw=none, from=0, to=1]
\end{tikzcd}\]
and the equivalence provided by Proposition \ref{nonunital} we get the wanted adjunction
\[\begin{tikzcd}
	{\mathrm{DLatt}_{lb}} & {\mathrm{CAlg}^{nu}(\mathrm{Sp})}
	\arrow[""{name=0, anchor=center, inner sep=0}, curve={height=-12pt}, from=1-1, to=1-2]
	\arrow[""{name=1, anchor=center, inner sep=0}, curve={height=-12pt}, from=1-2, to=1-1]
	\arrow["\dashv"{anchor=center, rotate=-90}, draw=none, from=0, to=1]
\end{tikzcd}\]
Tracing through the definitions we see that the left adjoint is given as $D \mapsto \mathrm{fib}( \mathcal{M}(D_\infty) \rightarrow \mathcal{M}(\mathbf{2}) )$, which agrees with $\mathcal{M}(D)$ by Lemma \ref{splitofmot}. The right adjoint is given by sending a non-unital $E_\infty$-ring $A$ to the subset of $\mathrm{Idem}(A_+)$ of those idempotents that map to $0 \in \mathbb{S}$ under the augmentation $A_+ = A \oplus \mathbb{S} \rightarrow \mathbb{S}$, in other words the set of idempotents of $A$.
\end{proof}

\begin{remark}
The above adjunction restricts to an adjunction
\[\begin{tikzcd}
	{\mathrm{DLatt}_{lb}} & \hspace{3ex} {\mathrm{CAlg}^{nu}(\mathrm{Sp})_{conn}}
	\arrow[""{name=0, anchor=center, inner sep=0}, "{\mathcal{M}}", curve={height=-12pt}, from=1-1, to=1-2]
	\arrow[""{name=1, anchor=center, inner sep=0}, "{\mathrm{Idem}}", curve={height=-12pt}, from=1-2, to=1-1]
	\arrow["\dashv"{anchor=center, rotate=-90}, draw=none, from=0, to=1]
\end{tikzcd}\]
where $\mathrm{CAlg}^{nu}(\mathrm{Sp})_{conn}$ is the $\infty$-category of connective non-unital $E_\infty$-rings. 
By composing with the adjunction
\[\begin{tikzcd}
	{\mathrm{CAlg}^{nu}(\mathrm{Sp})_{conn}} & {\mathrm{CAlg}^{nu}}
	\arrow["{\pi_0}", curve={height=-12pt}, from=1-1, to=1-2]
	\arrow[curve={height=-12pt}, hook, from=1-2, to=1-1]
\end{tikzcd}\]
we obtain the adjunction claimed in Theorem \ref{idempotents}.
\end{remark}

\section{Algebraic $K$-theory of (locally) coherent spaces}

We are now ready to prove the two main theorems of this paper.

\begin{theorem} \label{ktheorycoherent}
Let $X$ be a coherent space. Then the natural map $X^{const} \rightarrow X$ induces an equivalence
$$F^{cont}( \mathrm{Sh}(X,\mathcal{C})) \simeq F^{cont}( \mathrm{Sh}(X^{const},\mathcal{C})) $$
for any finitary localizing invariant $F$ and dualizable stable $\infty$-category $\mathcal{C}$.
\end{theorem}

\begin{proof}
Write $D = \mathcal{K}^o(X)$ for the bounded distributive lattice of compact open sets of $X$. The natural map $X^{const} \rightarrow X$ corresponds under Stone-Duality to the homomorphism
$D \rightarrow \mathrm{Bool}(D)$. We can write $D = \mathrm{colim}_{i \in I} D_i$ with $D_i$ finite frames, hence using Theorem \ref{sheavesfiltereddualizablelowerbounded}, it suffices to prove the theorem for the case of $D$ being a finite frame. Then we have the natural commuting square
\[\begin{tikzcd}
	{F( \mathrm{Sh}(D,\mathcal{C}))} & {F( \mathrm{Sh}(\mathrm{Bool}(D),\mathcal{C}))} \\
	{F(\mathcal{C})^{|\mathrm{pts}(D)|}} & {F(\mathcal{C})^{|\mathrm{pts}(\mathrm{Bool}(D))|}}
	\arrow[from=1-1, to=1-2]
	\arrow["\simeq", from=2-1, to=1-1]
	\arrow["\simeq", from=2-1, to=2-2]
	\arrow["\simeq"', from=2-2, to=1-2]
\end{tikzcd}\]
where the bottom map is an equivalence since $\mathrm{pts}(D) \cong \mathrm{pts}(\mathrm{Bool}(D))$ and the vertical maps are equivalences by Example \ref{sheavesfiniteframepresentable} and Corollary \ref{semiorthogonaldecomposition}.
\end{proof}

Given Theorem \ref{ktheorycoherent}, we will now extend the statement to locally coherent spaces. Note that if $\mathcal{C}$ is a presentable stable $\infty$-category, it is tensored over the $\infty$-category of spectra $\mathrm{Sp}$. For $A \in \mathrm{Sp}$ and $c \in \mathcal{C}$, we write $A \otimes c$ for their tensor. Also recall the existence of the spectrum-of-motives-functor $\mathcal{M}$ provided by Theorem \ref{spaceofidempotents}.

\begin{theorem} \label{maintheorem}
Let $X$ be a locally coherent space corresponding to the frame $\mathrm{Ind}(D)$, where $D$ is the lower bounded distributive lattice of compact open subsets of $X$. Then for any finitary localizing invariant $F$ and dualizable $\infty$-category $\mathcal{C}$ there is a natural equivalence
$$F( \mathrm{Sh}(X; \mathcal{C}) ) \simeq \mathcal{M}(D) \otimes F(C).$$
If the localizing invariant $F$ takes values in the $\infty$-category of spectra, we have that
$$\pi_n F( \mathrm{Sh}(X; \mathcal{C}) ) \cong M(D) \otimes_\mathbb{Z} \pi_n F(C).$$
\end{theorem}

\begin{proof} Let us first prove the case for $X$ being profinite, with $X = \lim_{i \in I} X_i$ for $X_i$ finite discrete sets. In this case we have
$$\mathrm{Sh}(X;\mathcal{C}) \simeq \colim_{i \in I} \mathrm{Sh}(X_i;\mathcal{C}) \simeq \colim_{i \in I} \bigoplus_{X_i}\mathcal{C} $$
in $\mathrm{Cat}^{dual}$ by Theorem \ref{sheavesfiltereddualizablelowerbounded}. Hence applying $F$ we get
$$F(\mathrm{Sh}(X;\mathcal{C})) \simeq \colim_{i \in I} \bigoplus_{X_i} F(\mathcal{C}) \simeq (\colim_{i \in I} \bigoplus_{X_i} \mathbb{S}) \otimes F(\mathcal{C}) \simeq \mathcal{M}(\mathcal{K}^o(X))  \otimes F(\mathcal{C}). $$

Now assume that $X$ is locally coherent space corresponding to the frame $\mathrm{Ind}(D)$. We can reduce to the case of coherent spaces by using the Verdier sequence
$$\mathrm{Sh}((D,fin);\mathcal{C}) \rightarrow \mathrm{Sh}((D_\infty,fin);\mathcal{C}) \rightarrow \mathcal{C}$$
given by the open-closed decomposition of $\mathrm{Ind}(D_\infty)$, see Display \ref{openclosedtopelement}, together with the natural splitting $M(D_{\infty}) \cong M(D) \oplus \mathbb{Z}$ provided by \ref{splitofmot}. The case of coherent spaces reduces to the profinite case by Theorem \ref{ktheorycoherent} and Proposition \ref{motivesonbooleanization}. The statement about the isomorphisms
$$\pi_n F( \mathrm{Sh}(X; \mathcal{C}) ) \cong M(D) \otimes_\mathbb{Z} \pi_n F(C)$$
in the case that $F$ has values in spectra follows directly from the construction of $\mathcal{M}(D)$, using the fact that taking homotopy groups commutes with filtered colimits.
\end{proof}

\begin{remark}
Note that if $X_i \rightarrow X_j$ is a transition map of finite sets in the diagram describing $X =  \lim_{i \in I} X_i $, then we have a map of spectra $\bigoplus_{X_i}\mathbb{S} \rightarrow \bigoplus_{X_j}\mathbb{S}$ naturally. The corresponding transition maps 
$$\bigoplus_{X_i} F(\mathcal{C}) \leftarrow \bigoplus_{X_j}F(\mathcal{C})$$
for the colimit in question are obtained by applying the functor $\mathrm{map}(-, F(\mathcal{C}))$.
\end{remark}

\begin{corollary} \label{ktheorycoherentvaluesspectra}
Let $X$ be a coherent space and $\mathcal{C}$ a dualizable $\infty$-category. Assume that $F$ is a localizing invariant with values in spectra. Then for all $n \in \mathbb{Z}$ the natural map $X^{const} \rightarrow X$ induces isomorphisms
$$ \pi_n F^{cont}( \mathrm{Sh}(X,\mathcal{C})) \cong C(X^{const}; \pi_n(F^{cont}(\mathcal{C})) ),$$
where $\pi_n(F^{cont}(\mathcal{C}))$ is equipped with the discrete topology.
\end{corollary}

\begin{proof}
This follows directly from the isomorphism $M(\mathcal{K}^o(X^{const})) \cong C(X^{const}; \mathbb{Z})$ provided by Proposition \ref{motivesforboolean}.
\end{proof}

\section{Applications}

\subsection{Scissors Congruence of polytopes} \label{scissorscongr}

An interesting application of the theory appears in the study of scissors congruence. Let $X$ be a geometry, such as Euclidean space $E^n$, the sphere $S^n$ or hyperbolic space $H^n$, of dimension $n$. A geometric $n$-simplex is the convex hull of $n+1$ points of non-trivial measure. (Meaning the $n+1$-points do not lie on an $n-1$-dimensional subspace.) A polytope in $X$ is a finite union of geometric $n$-simplices. (Note that this includes the empty set $\emptyset$.) Let $G$ be a subgroup of the group of isometries. A central question is Hilbert's third problem.

\begin{question}[Hilbert's third problem]
Given two polytopes $P$ and $Q$, when can $P$ be cut up into finitely many pieces, moved around via isometries in $G$ and reassembled into $Q$?
\end{question}

The question can be phrased differently. Consider the free abelian group $P(X,G)$ generated by the polytopes in $X$ modulo the relations
\begin{itemize}
\item $[P \cup Q] = [P] + [Q]$ whenever the intersection of $P$ and $Q$ has measure zero,
\item $[P] = [gP]$ for any isometry $g \in G$.
\end{itemize}
The group $P(X,G)$ is called \emph{scissors congruence group} of $X$. Hilbert's third problem can be reinterpreted as the question of finding an explicit description of this group (via invariants such as volume or the Dehn invariant).

Inna Zakharevich has generalized this abelian group into a spectrum $K^{sci}(X,G)$ called Scissor's congruence $K$-theory, whose zero-th homotopy group agrees with $P(X,G)$ and whose higher homotopy groups provide a playing field for the question of finding ``higher Dehn invariants'', see \cite{zakharevich2011scissorscongruencektheory}, also \cite{malkiewich2024higherscissorscongruence} and \cite{kupers2024scissorsautomorphismgroupshomology}. Let us mention two central statements about the description of this spectrum. The first is about the non-equivariant case, i.e. when we let $G = \{1\}$ be our group of isometries.

\begin{theorem}[Malkiewich-Zakharevich, \cite{malkiewich2024higherscissorscongruence} Theorem 1.10]
For any geometry $X$, there is an equivalence
$$K^{sci}(X,1) \simeq \bigoplus \mathbb{S},$$
where the wedge product is indexed by a choice of basis of the free (!) abelian group $P(X,1)$.
\end{theorem}

Furthermore, the spectrum $K^{sci}(X,1)$ naturally inherits an action by any group of isometries $G$. This is used in the following theorem.

\begin{theorem}[Bohmann-Gerhardt-Malkiewich-Merling-Zakharevich, \cite{malkiewich2024higherscissorscongruence} Theorem 1.11] \label{scissorsequivariant}
There is an equivalence of spectra
$$K^{sci}(X,1)_{hG} \simeq K^{sci}(X,G).$$
\end{theorem}

This connects to the statements in this paper as follows. Note that the collection of polytopes in $X$ forms a lower bounded distributive lattice $D(X)$, where joins are given by union and meets are given by intersection up to measure zero (so two polytopes that touch along a set of $(n-1)$-faces have empty meet). It is clear by definition that
$$M(D(X)) \cong P(X,1).$$
The latter group is also referred to as the polytope module of $X$. Associated to $D(X)$, we have a locally coherent space $X_{\mathrm{Poly}}$ under Stone duality, and a continuous map $X_{\mathrm{Poly}} \rightarrow X$. (The inverse image part sends an open $U$ of $X$ to the ideal of polytopes contained in $U$.) An immediate corollary of Theorem \ref{maintheorem} is the following, by noting that $\mathrm{THH}(\mathrm{Sp}) \simeq \mathbb{S}$.

\begin{corollary}
Let $X$ be a geometry. There is an equivalence
$$\mathrm{THH}( \mathrm{Sh}( X_{\mathrm{Poly}}; \mathrm{Sp} ) ) \simeq \mathcal{M}(D(X)) \simeq K^{sci}(X,1).$$
\end{corollary}

If $G$ is a group of isometries, there is a natural action of $G$ on $X_{\mathrm{Poly}}$ via homeomorphisms (since $G$ acts on $D(X)$ via lattice isomorphisms). Theorem \ref{scissorsequivariant} immediately implies  the following statement.

\begin{corollary}
Let $X$ be a geometry and $G$ a group of isometries of $X$. There is an equivalence
$$\mathrm{THH}( \mathrm{Sh}( X_{\mathrm{Poly}}; \mathrm{Sp} ))_{hG}  \simeq K^{sci}(X,G).$$
\end{corollary}

This corollary means that higher scissors congruence $K$-theory can be thought of as a special case of the study of topological Hochschild Homology, or equivalently traces, of a class of dualizable $\infty$-categories. This tightens the relationship between scissors congruence $K$-theory with algebraic $K$-theory and could explain why several techniques have close analogues on both sides. However, there is more to be done. For one, it is unsatisfactory that $G$-orbits need to be taken after applying $\mathrm{THH}$. It would be interesting to figure out the direct relationship of Scissors Congruence $K$-theory with the $\infty$-topos of $G$-equivariant sheaves on $X_{\mathrm{Poly}}$. This would open the door to an understanding of other related spectra, such as the Scissors Congruence $K$-theories of manifolds or varieties.

\begin{remark}
For a given polytope $P$, the lattice $D(X)_{/P}$ is in fact Boolean. This means that $X_{\mathrm{Poly}}$ has a neighborhood basis of profinite spaces, closed under intersection. As a consequence we see that $X_{\mathrm{Poly}}$ is a locally compact Hausdorff space. 
\end{remark}

\subsection{Algebraic $K$-theory of measure spaces} \label{algebraicktheorymeasure}

Suppose $(X, \mathcal{L}, \mu)$ is a localizable measure space in the sense of \cite[Definition 211G]{fremlin2000measure}. We remark that this includes the case of $\sigma$-finite measure spaces.  Then we can associate to $X$ the complete Boolean algebra $\mathcal{B} = \mathcal{L}/\mathcal{N}$ where $\mathcal{N}$ is the $\sigma$-ideal of $\mu$-null sets, as discussed in example \ref{measurespacemotive}. We call $\mathrm{Stone}(X)$ the associated profinite space to the Boolean algebra $\mathcal{B}$, or equivalently described by the frame $\mathrm{Ind}(\mathcal{B})$.

Since $\mathcal{B}$ is Boolean, a valuation $\mu : \mathcal{B} \rightarrow \mathbb{C}$ is the same as a finitely additive complex measure on $X$, which vanishes on $\mu$-null sets. We define the \emph{total variation} of a $\mathbb{C}$-valued valuation $\mu$ on $\mathcal{B}$ to be
$$||\mu|| = \mathrm{sup}_{P_1,\hdots, P_n \text{ finite partition of } \mathcal{B}} \sum_{i = 1}^n |\mu(P_i)|.$$
We say that $\mu$ is bounded if its total variation is finite.

The dual space $L^\infty(X)^*$ to $L^\infty(X)$ can be characterized via an analogous statement to the Riesz representation theorem as the space of bounded valuations on $\mathcal{B}$, a result that is referred to as the Yoshida-Hewitt representation. For a proof see \cite{Toland2020}, at least in the case when $X$ is a complete, $\sigma$-finite measure space. Phrased differently, Theorem \ref{maintheorem} implies that we have an inclusion
$$L^\infty(X)^* \subset \mathrm{Hom}( K_0( \mathrm{Sh}(\mathrm{Stone}(X);\mathrm{Sp}) ), \mathbb{C} ).$$
that identifies $L^\infty(X)^*$ with the bounded elements in the target.

The module of motives of $\mathcal{B}$ also appears naturally. As seen in Example \ref{measurespacemotive}, we can identify 
$$K_0( \mathrm{Sh}(\mathrm{Stone}(X);\mathrm{Sp}) ) \cong M(\mathcal{B}) \subset L^\infty(X)$$
with the subset given by integral step functions on $X$.

Now if $(X,\mu)$ is a localizable measure space, the $*$-algebra $L^\infty(X,\mu)$ is a commutative von Neumann algebra, which is a fact going back to Segal \cite{SegalI.E.1951EoMS}, see also \cite[Theorem 243G]{fremlin2000measure}. The commutative von Neumann algebra $L^\infty(X,\mu)$ also has two associated $K$-theory spectra given by the algebraic $K$-theory $K^{alg}(L^\infty(X,\mu))$, where we simply consider $L^\infty(X,\mu)$ as a ring, as well as the topological/operator $K$-theory $K^{top}(L^\infty(X,\mu))$ with $L^\infty(X,\mu)$ viewed as a $C^*$-algebra. Note that for any $C^*$-algebra $A$, there is an isomorphism $K^{alg}_0(A) \cong K^{top}_0(A)$. The computation of the algebraic $K$-theory of von Neumann algebras in degree $0$ and $1$ has been done by Lück \cite{Lück2002}, Lück-R{\o}rdam \cite{Lueck1993}. See also \cite{reich_operator}. In the commutative case, this computation reduces to the following statements. 
\begin{proposition}
Let $(X,\mu)$ be a localizable measure space. Then
$$\begin{array}{rcl} K^{alg}_0( L^\infty(X;\mu) ) &\cong & L^\infty(X;\mu; \mathbb{Z}) \\
K^{alg}_1( L^\infty(X;\mu) ) &\cong & L^\infty(X;\mu)^\times,
\end{array}$$
where $L^\infty(X; \mathbb{Z})$ are the essentially bounded functions with values in $\mathbb{Z}$ and and $L^\infty(X)^\times$ denotes the abelian group of units in $L^\infty(X)$.
\end{proposition}

For the statement about $K_0$, see \cite[Theorem 9.13 and Example 9.14.]{Lück2002}. The statement about $K_1$ can be found in \cite[Theorem 2.1]{Lueck1993}. 

Using the fact that $K_0(\mathbb{S}) = \mathbb{Z}$ and $K_1(\mathbb{S}) = \{\pm 1\}$, this means that for a localizable measure space $(X,\mu)$ we have natural inclusions
$$\begin{array}{rcccccl} K_0( \mathrm{Sh}(\mathrm{Stone}(X);\mathrm{Sp}) ) &\cong & M(\mathcal{B}) &\subset & L^\infty(X; \mathbb{Z})  &\cong & K^{alg}_0( L^\infty(X;\mu) )  \\
K_1( \mathrm{Sh}(\mathrm{Stone}(X);\mathrm{Sp}) ) & \cong & M(\mathcal{B}) \otimes_\mathbb{Z} \{\pm 1\} & \subset & L^\infty(X)^\times &\cong & K^{alg}_1( L^\infty(X;\mu) ).
\end{array}$$
which correspond to the inclusions of functions taking on finitely many values, or functions taking values in $\{\pm 1\}$, respectively.

In summary, the relation between the (algebraic) $K$-theory of $\mathrm{Stone}(X)$ and the operator theory of $L^\infty(X,\mu)$ is very tight, related by forms of completions or passage to bounded elements. This suggests that one should think of the compactly generated category  $\mathrm{Sh}(\mathrm{Stone}(X);\mathrm{Sp})$ as the corresponding object in the setting of dualizable $\infty$-categories to the commutative von Neumann algebra $L^\infty(X,\mu)$ in the setting of $C^*$-algebras.

This analogy can be made more formal, using the notion of measurable locales, due to Pavlov \cite{Pavlov_2022}. A locale $L$ is \emph{measurable}, if its frame is Boolean and localizable, which means roughly speaking that it admits sufficiently many continuous valuations (See \cite[Definition 2.52]{Pavlov_2022}). There is the chain of equivalences of categories.
\begin{theorem}[\cite{Pavlov_2022} Theorem 1.1.]
The following categories are equivalent:
\begin{itemize}
\item The category $\mathrm{CSLEMS}$ of compact strictly localizable enhanced measurable spaces.
\item The category $\mathrm{HStonean}$ of hyperstonean spaces and open maps.
\item The category $\mathrm{HStoneanLoc}$ of hyperstonean locales and open maps.
\item The category $\mathrm{MeasLoc}$ given by the full subcategory of the category of locales spanned by the measurable locales.
\item The category $\mathrm{CVNA}^{op}$ opposite to the category of commutative von Neumann algebras .
\end{itemize}
\end{theorem}

A compact strictly localizable enhanced measurable space $(X, \mathcal{L}, \mathcal{N})$ is a set $X$, a $\sigma$-algebra $\mathcal{L}$ on $X$ and a $\sigma$-ideal $\mathcal{N}$ of negligible subsets of $X$, satisfying additional conditions, which in particular imply that the Boolean algebra $\mathcal{B} = \mathcal{L}/\mathcal{N}$ is complete. (An example is the measure space obtained from a Radon measure on a topological space.) Under these equivalences $X$ is sent to the measurable locale with frame given by $\mathcal{L}/\mathcal{N}$, or the corresponding Stone space $\mathrm{Stone}(X)$ with the Boolean algebra of compact opens given by $\mathcal{L}/\mathcal{N}$. The corresponding commutative von Neumann algebra is $L^\infty(X)$. Thus we have a fitting notion of algebraic $K$-theory of a measure space, simply as the algebraic $K$-theory of the corresponding hyperstonean space.

\appendix

\begingroup
\setlength{\emergencystretch}{8em}
\printbibliography
\endgroup

\end{document}